\documentclass{amsart}

\usepackage{amsmath, amssymb, amsthm, amsfonts, color}

\input xy
\xyoption{all}

\usepackage{hyperref}

\swapnumbers
\numberwithin{equation}{section}

\theoremstyle{plain}
\newtheorem{theorem}[subsection]{Theorem}
\newtheorem{lemma}[subsection]{Lemma}
\newtheorem{prop}[subsection]{Proposition}
\newtheorem{cor}[subsection]{Corollary}
\newtheorem{conj}[subsection]{Conjecture}

\theoremstyle{definition}
\newtheorem{defn}[subsection]{Definition}

\newtheorem{remark}[subsection]{Remark}
\newtheorem{exam}[subsection]{Example}
\newtheorem{ques}[subsection]{Question}

\def\BB{\mathbb{B}}
\def\CC{\mathbb{C}}

\def\FF{\mathbb{F}}
\def\GG{\mathbb{G}}

\def\LL{\mathbb{L}}

\def\NN{\mathbb{N}}

\def\QQ{\mathbb{Q}}

\def\ZZ{\mathbb{Z}}

\newcommand\cA{\mathcal{A}}
\newcommand\cB{\mathcal{B}}

\newcommand\cE{\mathcal{E}}

\newcommand\cH{\mathcal{H}}

\newcommand\cK{\mathcal{K}}
\newcommand\cL{\mathcal{L}}
\newcommand\cM{\mathcal{M}}
\newcommand\cN{\mathcal{N}}
\newcommand\cO{\mathcal{O}}
\newcommand\cP{\mathcal{P}}
\newcommand\cQ{\mathcal{Q}}

\newcommand\cV{\mathcal{V}}

\newcommand\cZ{\mathcal{Z}}

\def\bI{\mathbf{I}}

\def\bP{\mathbf{P}}

\def\bR{\mathbf{R}}

\newcommand\frB{\mathfrak{B}}

\newcommand\frD{\mathfrak{D}}

\newcommand\frb{\mathfrak{b}}
\newcommand\frc{\mathfrak{c}}

\newcommand\frg{\mathfrak{g}}

\newcommand\frl{\mathfrak{l}}
\newcommand\fm{\mathfrak{m}}

\newcommand\frq{\mathfrak{q}}

\newcommand\tilW{\widetilde{W}}

\newcommand{\Bun}{\textup{Bun}}

\newcommand{\ch}{\textup{char}}

\newcommand\ev{\textup{ev}}

\newcommand{\Fl}{\textup{Fl}}
\newcommand{\fl}{f\ell}

\newcommand{\Gr}{\textup{Gr}}

\newcommand{\Hk}{\textup{Hk}}

\newcommand\IC{\textup{IC}}
\newcommand\id{\textup{id}}
\renewcommand{\Im}{\textup{Im}}
\newcommand{\Ind}{\textup{Ind}}

\newcommand\Irr{\textup{Irr}}

\newcommand\Lie{\textup{Lie}\ }

\newcommand\loc{\textup{loc}}

\newcommand\Perv{\textup{Perv}}
\newcommand{\Pic}{\textup{Pic}}
\newcommand\pr{\textup{pr}}

\newcommand\pt{\textup{pt}}

\newcommand{\reg}{\textup{reg}}
\newcommand{\rel}{\textup{rel}}

\newcommand\rs{\textup{rs}}

\newcommand\Spec{\textup{Spec}\ }
\newcommand\St{\textup{St}}

\newcommand\Sym{\textup{Sym}}

\newcommand{\Tr}{\textup{Tr}}

\newcommand{\val}{\textup{val}}

\newcommand\Hom{\textup{Hom}}

\newcommand\GL{\textup{GL}}

\newcommand\SL{\textup{SL}}
\renewcommand\sl{\mathfrak{sl}}

\newcommand{\Gm}{\GG_m}

\newcommand{\Ad}{\textup{Ad}}

\newcommand\xch{\mathbb{X}^*}
\newcommand\xcoch{\mathbb{X}_*}

\newcommand{\incl}{\hookrightarrow}
\newcommand{\isom}{\stackrel{\sim}{\to}}
\newcommand{\bij}{\leftrightarrow}
\newcommand{\surj}{\twoheadrightarrow}
\newcommand{\leftexp}[2]{{\vphantom{#2}}^{#1}{#2}}
\newcommand{\pH}{\leftexp{p}{\textup{H}}}

\newcommand{\Ql}{\QQ_{\ell}}
\newcommand{\Qlbar}{\overline{\QQ}_\ell}

\newcommand{\wt}[1]{\widetilde{#1}}
\newcommand{\wh}[1]{\widehat{#1}}
\newcommand\quash[1]{}

\newcommand{\bu}{\bullet}
\newcommand{\ov}{\overline}
\newcommand{\bs}{\backslash}
\newcommand{\tl}[1]{[\![#1]\!]}
\newcommand{\lr}[1]{(\!(#1)\!)}

\newcommand\xr{\xrightarrow}
\newcommand\op{\oplus}
\newcommand\ot{\otimes}

\newcommand{\sslash}{\mathbin{/\mkern-6mu/}}
\renewcommand\c\circ
\newcommand{\mt}{\mapsto}
\newcommand{\bt}{\boxtimes}

\newcommand{\homog}[2]{\textup{H}_{#1}({#2})}  
\newcommand{\cohog}[2]{\textup{H}^{#1}({#2})}     
\newcommand{\cohoc}[2]{\textup{H}_{c}^{#1}({#2})}     
\newcommand{\hBM}[2]{\textup{H}^{\textup{BM}}_{#1}({#2})}  

\newcommand\upH{\textup{H}}

\renewcommand\a\alpha
\renewcommand\b\beta
\newcommand\g\gamma
\newcommand\G\Gamma
\renewcommand\d\delta
\newcommand\D\Delta

\newcommand{\io}{\iota}
\renewcommand{\th}{\theta}

\newcommand{\ph}{\varphi}
\renewcommand\r\rho

\newcommand{\Sig}{\Sigma}
\renewcommand{\t}{\tau}

\newcommand{\y}{\eta}
\newcommand{\z}{\zeta}

\renewcommand{\l}{\lambda}
\renewcommand{\L}{\Lambda}
\renewcommand\k{\kappa}
\newcommand{\om}{\omega}

\newcommand\hs{\heartsuit}
\newcommand\na{\natural}

\newcommand{\Wa}{W_{\mathrm{aff}}}

\newcommand\vn{\varnothing}
\newcommand\Spr{\textup{Spr}}
\newcommand\RT{\textup{RT}}
\newcommand\el{\textup{ell}}
\newcommand\pur{\textup{pur}}
\newcommand{\Hit}{\textup{Hit}}

\title{Perverse filtration on the cohomology of affine Springer fibers}

\dedicatory{Dedicated to George Lusztig's 80th birthday}
\author{Zhiwei Yun}
\thanks{Supported partially by the Simons Foundation.}
\thanks{{\bf Disclaimer: the author does not use AI in this research.}}
\address{Department of Mathematics, Massachusetts Institute of Technology, 77 Massachusetts Ave, Cambridge, MA 02139}
\email{zyun@mit.edu}
\date{}
\subjclass[2020]{20G99}
\keywords{Affine Springer fibers, perverse filtration}

\begin{document}

\begin{abstract}
We show that the top cohomology of any affine Springer fiber, as a Weyl group representation, contains a large part of the total cohomology of certain Springer fibers. The main ingredient of the proof is the construction of a ``perverse filtration'' on the pure part of the cohomology of affine Springer fibers.
\end{abstract}

\maketitle

\section{Introduction}

\quash{consider adding:

DAHA duality for dual affine Springer fibers, local version of \cite{YLD}. This makes better sense if there's canonicity of perverse filtration.

Canonicity of perverse filtration?

Parabolic induction for $E_{\g}$. Weaker: as $W$-mod; stronger as $\tilW$-module.

Conjecture on strata (not sure if it fits).

}

\subsection{Affine Springer fibers and affine Weyl group actions}
Let $G$ be simply-connected, almost simple algebraic group over an algebraically closed field $k$ with Lie algebra $\frg$. We assume $p=\ch(k)>2h$, where $h$ is the Coxeter number of $G$. Fix a maximal torus and a Borel subgroup $T\subset B\subset G$, and let $W=W(G,T)$ be the Weyl group.  Let $\cB\cong G/B$ be the flag variety of $G$. 

Let $F=k\lr{t}$ be the field of Laurent series over $k$. Let $LG$ be the loop group of $G$ whose $k$-points are $G(F)$, and $L^+G$ be the arc group of $G$ whose $k$-points are $G(\cO_{F})$, where $\cO_{F}=k\tl{t}$. Let $\bI\subset L^+G$ be the Iwahori subgroup, i.e., preimage of $B$ under the reduction map $L^+G\to G$. Let $\Wa=\xcoch(T)\rtimes W$ be the affine Weyl group. Let $\Gr=LG/L^{+}G$ be the affine Grassmannian of $G$ and $\Fl=LG/\bI$ be the affine flag variety.  
 
Let $\frg\lr{t}=\frg\ot F$. A regular semisimple element $\g\in \frg\lr{t}$ is called integral (resp. topologically nilpotent) if for any $f\in k[\frg]^{G}$ homogeneous of positive degree, we have $f(\g)\in  \cO_{F}$ (resp. $f(\g)\in  t \cO_{F}$). For an integral regular semisimple $\g\in \frg\lr{t}$, its affine Springer fiber as defined by Kazhdan and Lusztig \cite{KL} is the following reduced closed sub-ind-scheme of $\Fl$
\begin{equation*}
\Fl_{\g}=\{g\bI\in \Fl|\Ad(g^{-1})\g\in \Lie \bI\}.
\end{equation*}
By \cite[\S4, Proposition 1]{KL}, $\Fl_{\g}$ is equidimensional. Let $d_{\g}=\dim \Fl_{\g}$. Similarly we have the affine Springer fiber $\Gr_{\g}\subset \Gr$, which also has dimension $d_{\g}$.

Let $LG_{\g}$ be the centralizer of $\g$ under the loop group $LG$. This is the loop group attached to the centralizer $G_{\g}=C_{G_{F}}(\g)$, which is a maximal torus of $G_{F}$ over $F$. Let
\begin{equation*}
A_{\g}=\pi_{0}(LG_{\g}),
\end{equation*}
a finitely generated abelian group whose rank is the split rank of $G_{\g}$ over $F$. The action of $LG_{\g}$ on $\Fl_{\g}$ induces an action of $A_{\g}$ on the \'etale cohomology $\cohog{*}{\Fl_{\g},\Qlbar}$, where $\ell$ is a prime number different from $\ch(k)$. See \S\ref{ss:coho} for our convention for cohomology of ind-schemes. Let
\begin{equation*}
E_{\g}=\cohog{2d_{\g}}{\Fl_{\g},\Qlbar}^{A_{\g}}.
\end{equation*}
Lusztig \cite{L} constructed an action of $\Wa$ on $\cohog{2d_{\g}}{\Fl_{\g}}$ that commutes with the $A_{\g}$-action; it is an affine analog of the Springer action. In particular we have an action of $\Wa$ on $E_{\g}$. This note concerns the structure of $E_{\g}$ as a $W$-module.

\subsection{Main results}

Let $e\in \frg$ be a nilpotent element and $\cB_{e}$ be its Springer fiber. Let $\g\in e+t\frg\tl{t}$ be a generic lifting of $e$ to $\frg\tl{t}=\frg\ot_{k}\cO_{F}$. In this case, Lusztig observed in many examples that $E_{\g}$ as a $W$-module is isomorphic to the $A_{e}$-invariants of the total cohomology of the Springer fiber $\cB_{e}$, i.e., $\cohog{*}{\cB_{e}}^{A_{e}}$, where $A_{e}=\pi_{0}(G_{e})$. A precise prediction for $e$ distinguished is formulated in Lusztig \cite[Last paragraph on paper {[80]}]{LComm}. One main result of this paper is a statement of this sort that is both weaker in the above case (i.e., $\g$ is a generic lifting of a nilpotent element) and more general (i.e., it makes sense for any topologically nilpotent element $\g$).    

\begin{theorem}\label{th:intro main} Let $\cO$ be a nilpotent orbit of $\frg$ and $\g$ be a regular semisimple and topologically nilpotent element of $\frg\lr{t}$. Suppose the adjoint orbit of $\g$ intersects $\cO+t\frg\tl{t}$ (i.e., $\cO$ is a {\em reduction type} of $\g$). Then $E_{\g}$ contains the $W$-module ${}'H_{\cO}:=\Im(\cohog{*}{\cB}\to \cohog{*}{\cB_{e}}^{A_{e}})$ as a subquotient, where $e\in \cO$. 
\end{theorem}

Here is a more precise statement. Polynomials in degree $2$ classes on the affine flag variety give a ring homomorphism by restricting to $\Fl_{\g}$
\begin{equation}\label{intro ph}
\Sym^{d_{\g}}(\cohog{2}{\Fl})\to \cohog{2d_{\g}}{\Fl}\xr{i_{\g}^{*}} \cohog{2d_{\g}}{\Fl_{\g}}^{A_{\g}}=E_{\g}.
\end{equation}
Let ${}'E_{\g}$ be the image of the above map. We equip ${}'E_{\g}$ with a decreasing filtration $Q^{\bu}({}'E_{\g})$ as follows. We have a natural decomposition 
\begin{equation*}
\cohog{2}{\Fl}=\cohog{2}{\cB}\op \cohog{2}{\Gr}
\end{equation*}
where $\cohog{2}{\Gr}$ is one-dimensional. Define $Q^{j}('E_{\g})$ to be the image of $(\cohog{2}{\cB})^{j}\cap \Sym^{d_{\g}}(\cohog{2}{\Fl})$ (where $(\cohog{2}{\cB})^{j}$ is the $j^{\textup{th}}$ power of the ideal $(\cohog{2}{\cB})\subset \Sym(\cohog{2}{\Fl})$). Then in Theorem \ref{th:main} we show that the associated graded $\Gr^{*}_{Q}{}'E_{\g}$ admits a canonical $W$-equivariant graded (up to a scaling by $2$) surjection
\begin{equation*}
\Gr^{*}_{Q}{}'E_{\g}\surj {}'H^{2*}_{\cO}.
\end{equation*}

We derive some consequences of the above theorem on the reduction types of $\g$. Recall that a nilpotent orbit $\cO$ of $G$ is called a reduction type of $\g$ if the adjoint orbit of $\g$ intersects $\cO+t\frg\tl{t}$. Each reduction type $\cO$ of $\g$ defines a locally closed sub-ind-scheme $\Gr_{\g,\cO}$ in the affine Springer fiber $\Gr_{\g}$: it consists of $gL^{+}G$ such that $\Ad(g^{-1})\g\in \cO+t\frg\tl{t}$. By considering the preimage of $\Gr_{\g,\cO}$ in $\Fl_{\g}$, we have a trivial inequality
\begin{equation}\label{dim GrO}
\dim\Gr_{\g,\cO}\le d_{\g}-d_{\cO}
\end{equation}
where $d_{\cO}$ is the dimension of the Springer fiber $\cB_{e}$ for any $e\in \cO$.

\begin{cor} Let $\g$ be a topologically nilpotent element of $\frg\lr{t}$.
\begin{enumerate}
\item The equality in \eqref{dim GrO} holds for every reduction type of $\g$. In other words, the preimage of $\Gr_{\g,\cO}$ in $\Fl_{\g}$ has top dimension.
\item The union of nilpotent orbits that are reduction types of $\g$ is open. 
\end{enumerate}
\end{cor}

We also give a strengthening of a theorem of Cheng-Chiang Tsai \cite{Tsai}:

\begin{prop}[See Proposition \ref{p:reg rep}]
Suppose $\g\in\frg\lr{t}$ is regular semisimple and $t^{-1}\g$ is integral. Then as a $\Wa$-module, $E_{\g}$ is {\em dual} to the space of $W$-harmonic polynomials on $V=\cohog{2}{\cB}$, where the lattice and finite Weyl group part of $\Wa$ act on $V$ by translation and the reflection representation respectively. In particular, as a $W$-module, $E_{\g}$ is isomorphic to the regular representation.
\end{prop}

Finally, when $G$ is of type $A$, a result of Kivinen and Tsai \cite{KT} gives an optimal version of Theorem \ref{th:intro main}, namely,

\begin{theorem}[see Theorem \ref{th:A}]\label{th:intro A} Let $G=\SL_{n}$ and $\g\in\frg\lr{t}$ be a regular semisimple and topologically nilpotent element. Let $\cO$ be the minimal reduction type of $\g$. Then $E_{\g}={}'E_{\g}$ and there is a canonical isomorphism of graded $W$-modules
\begin{equation*}
\Gr^{*}_{Q}E_{\g}\cong \cohog{2*}{\cB_{e}}, \quad e\in \cO.
\end{equation*}
\end{theorem}

\subsection{Perverse filtration}

The proof of Theorem \ref{th:intro main} crucially relies on construction of a {\em perverse filtration} on (part of) the cohomology of $\Fl_{\g}$. Such a perverse filtration has been constructed in the special case where $\g$ is homogeneous in \cite{OY}. It is a local analog of the perverse filtration on the cohomology of Hitchin fibers coming from the Hitchin fibration. Section \ref{s:perv fil} is devoted to the construction of such a perverse filtration with desired properties using auxiliary global data. 

Let 
\begin{equation*}
H_{\g,\pur}=\bigoplus_{i\in \ZZ}\Gr_{W}^{i}\cohog{i}{\Fl_{\g}}^{A_{\g}}
\end{equation*}
be the pure part of the cohomology groups $\cohog{i}{\Fl_{\g}}^{A_{\g}}$, in the sense of either mixed Hodge theory (when $k=\CC$) or Frobenius weights (when $k=\ov\FF_{q}$). Note that the degree $i$ piece $H^{i}_{\g,\pur}$ is naturally a quotient of $\cohog{i}{\Fl_{\g}}^{A_{\g}}$. 

\begin{theorem}\label{th:intro perv fil} For each regular semisimple and integral element $\g\in \frg\lr{t}$, there exists an increasing filtration $P_{\bu}H_{\g,\pur}$ on $H_{\g,\pur}$ satisfying the axioms of a perverse filtration in Definition \ref{d:perv fil}.
\end{theorem}
The most powerful property of a perverse filtration is the hard Lefschetz property with respect to a nonzero class in $\cohog{2}{\Gr}$. The hard Lefschetz property is responsible for showing that $E_{\g}$ is ``large enough''.

The decreasing filtration $Q^{\bu}$ on $'E_{\g}$ introduced after Theorem \ref{th:intro main} is related to the perverse filtration on $E_{\g}$ (as the top degree piece of $H_{\g,\pur}$) in the following way
\begin{equation}\label{fil Q and P}
Q^{j}({}'E_{\g})={}'E_{\g}\cap P_{2d_{\g}-j}E_{\g}.
\end{equation}
 
It is of theoretical importance to see if the perverse filtration we define is actually independent of the global data, but we do not address this question here.


\subsection*{Acknowledgments}
The author would like to thank Oscar Kivinen, George Lusztig and Cheng-Chiang Tsai for helpful discussions.  He thanks Roman Bezrukavnikov and Minh-Tam Trinh for comments. He also thanks the referee for catching a gap in the argument in an earlier version. It is a pleasure to dedicate this short paper to the 80th birthday of George Lusztig.

\section{Cohomology and perverse filtration}
In this paper all \'etale cohomology groups will have $\Qlbar$-coefficients.

\subsection{Cohomology of affine Springer fibers}\label{ss:coho} For an indscheme $X$ that is a union of finite-type closed subschemes over $k$ (such as $\Fl_{\g}$ or $\Fl$, or their locally closed sub-indschemes), the cohomology of $X$ is defined to be the limit
\begin{equation}\label{coho as lim}
\cohog{n}{X}=\varprojlim_{Z}\cohog{n}{Z}
\end{equation}
where the limit is taken over all finite-type subschemes of $X$ (or restricting to closed subschemes of finite type, since they are cofinal in all finite-type subschemes). 

We also define the Borel-Moore homology of $X$ as
\begin{equation*}
\hBM{n}{X}=\varinjlim_{Z}\hBM{n}{Z}
\end{equation*}
where $Z$ runs over all {\em closed} subschemes of $X$ of finite type over $k$. The usual homology $\homog{*}{X}$ of $X$ is defined similarly (where we can take the colimit over all finite type subschemes), and it coincides with the Borel-Moore homology if $X$ is ind-proper. Then canonically $\cohog{n}{X}\cong \Hom_{\Qlbar}(\homog{n}{X},\Qlbar)$. 

Applying to $X=\Fl$, it is well-known that $\cohog{*}{\Fl}$ is finite dimensional in each degree.


Let $\g\in \frg\lr{t}$ be regular semisimple and integral. The group $A_{\g}=\pi_{0}(LG_{\g})$ acts on $\cohog{*}{\Fl_{\g}}$. 
Denote the $A_{\g}$-invariants on cohomology by
\begin{equation}\label{defn H gamma}
H_{\g}:=\cohog{*}{\Fl_\g}^{A_\g}
\end{equation}
with graded pieces $H^{i}_{\g}$. By \cite[\S2, Proposition 1]{KL}, there is a lattice $\L_{\g}\subset LG_{\g}$ acting freely on $\Fl_{\g}$ with a finite-type fundamental domain. Therefore $H_{\g}$ is finite-dimensional.

\quash{Denote the $A_{g}$-coinvariants by
\begin{equation*}
H_\g:=\homog{*}{\Fl_\g}_{A_\g}
\end{equation*}
with the homological grading $H_\g=\op_{j\in \ZZ}H_{\g,j}$. We have $E_{\g}=H_{\g, 2d_\g}$. Then $H_{\g}$ and $H^{\g}$ are dual graded vector spaces of finite total dimension.}

The inclusion $i_\g: \Fl_\g\incl \Fl$ induces a restriction map on cohomology
\begin{equation*}
i_\g^*: \cohog{*}{\Fl}\to H_{\g}.
\end{equation*}


\subsection{Perverse filtration}
Since $G$ is simply-connected and almost simple, $\dim \cohog{2}{\Gr}=1$. Let $0\ne \y\in \cohog{2}{\Gr}$. When confusion is unlikely, we also denote the pullback of $\y$ to $\cohog{2}{\Fl}$ as well as its restriction to $\Fl_\g$ by $\y$.

We have a canonical ring isomorphism
\begin{equation*}
\cohog{*}{\Fl}\cong\cohog{*}{\cB}\ot \cohog{*}{\Gr}
\end{equation*}
Here the copy of $\cohog{*}{\cB}$ in $\cohog{*}{\Fl}$ is generated by the Chern classes of line bundles $\cO(\l)$ (for $\l\in \xch(T)$) that extends from $\cB$ to $\Fl$. In particular, $\cohog{*}{\cB}$ can be viewed as a subring of $\cohog{*}{\Fl}$ in a natural way.

Let $\g\in\frg\lr{t}$ be regular semisimple and integral. 
A graded subquotient $V=V'/V''$ of $\cohog{*}{\Fl_{\g}}$ is called {\em admissible} if its degree zero piece $V^{0}=\cohog{0}{\Fl_{\g}}=\Qlbar$ (note that $\Fl_{\g}$ is connected by \cite[\S4, Lemma 2]{KL}), and the graded subspaces $V''\subset V'$  
are stable under the cup product action of $\cohog{2}{\Fl}$ (under $i_{\g}^{*}$) and the $\Wa$-action. In particular, $V$ carries an action of $\Sym(\cohog{2}{\Fl})\rtimes\Wa$.

\begin{defn}\label{d:perv fil} Let $V=\op_{j\in \ZZ}V^{j}$ be an admissible graded subquotient of $\cohog{*}{\Fl_{\g}}$. An increasing filtration $P_\bu V=(P_{i}V)_{i\in \ZZ}$ on $V$
is called a {\em perverse filtration} if it satisfies the following conditions
\begin{enumerate}
\item\label{comp grading} The filtration $P_\bu V$ is compatible with the cohomological grading on $V$, i.e., 
\begin{equation*}
P_i V=\bigoplus_{j\in \ZZ}V^{j}\cap P_i V.
\end{equation*}
We denote $V^{j}\cap P_i V$ by $P_i V^{j}$.
\item\label{fil bounds} $P_{-1}V=0$ and $P_{i}V^{i}=V^{i}$ for all $i$. In particular, $P_{2d_{\g}}V=V$. 
\item\label{PW} Each $P_iV$ is stable under the action of $\Wa$.
\item\label{HL} Cupping with $\y\in \cohog{2}{\Gr}$ takes $P_i V$ to $P_{i+2} V$, and induces an isomorphism for every $0\le i\le d_\g$
\begin{equation*}
\cup\y^{i}: \Gr^P_{d_\g- i} V\isom \Gr^P_{d_\g+i}V.
\end{equation*}
\item\label{Chern half} The action of $\cohog{2}{\cB}$ on $V$ sends $P_{i}V$ to $P_{i+1}V$, for all $i\in \ZZ$.
\end{enumerate}
\end{defn}

Since $\y\in \cohog{2}{\Gr}-\{0\}$ is unique up to a nonzero scalar, the notion of a perverse filtration on $V$ is independent of the choice of $\y$.

\begin{remark} For applications in this paper, we only need a weaker version of property \eqref{PW} above, namely we only need each $P_iV$ to be stable under the action of the finite Weyl group $W$, as a subgroup of $\Wa$.
\end{remark}

We record a formal consequence of the axioms of a perverse filtration.
\begin{lemma}\label{l:fil lower bound}
Let $P_{\bu}V$ be a perverse filtration on an admissible graded subquotient $V$ of $\cohog{*}{\Fl_{\g}}$. Then for any $i,j\in\ZZ$
\begin{equation*}
\Gr^{P}_{i}V^{j}\ne 0 \Rightarrow \frac{j}{2}\le i\le j.
\end{equation*}
\end{lemma}
\begin{proof}
The inequality $i\le j$ follows from property \eqref{fil bounds}. For the inequality $j\le 2i$, we only need to consider the case $i\le d_{\g}$ for otherwise since $j\le 2d_{\g}$ and it holds automatically.   When $i\le d_{\g}$ and $\Gr^{P}_{i}V^{j}\ne0$, property \eqref{HL} implies $\Gr^{P}_{2d_{\g}-i}V^{j+2d_{\g}-2i}\ne0$. In particular, $j+2d_{\g}-2i\le 2d_{\g}$, i.e., $j\le 2i$.
\end{proof}

\subsection{Weights}\label{ss:wt}

We will use the notion of {\em weights} on cohomology groups such as $\cohog{*}{\Fl_{\g}}$, following Deligne \cite[\S13-14]{D-poids}. Let $Z$ be a scheme of finite type over $k$. When the base field $k=\CC$, Deligne's mixed Hodge theory \cite{D-HodgeII} equips $\cohog{i}{Z(\CC),\QQ}$ with a weight filtration $W_{\bu}\cohog{i}{Z(\CC),\QQ}$. 

For a general base field $k$, $Z$ is obtained by base change from a finite type scheme $\cZ$ over a finitely generated $\ZZ[1/\ell]$-algebra $R$. For a closed point $s\in \Spec R$ with finite residue field $\FF_{s}$, the $\Qlbar$-cohomology $\cohog{i}{\cZ_{\ov s},\Qlbar}$ of the geometric fiber $\cZ_{\ov s}=\cZ\ot_{R}\ov{\FF_{s}}$ carries an action of the geometric Frobenius of $\ov{\FF_{s}}/\FF_{s}$, giving a grading on $\cohog{i}{\cZ_{\ov s},\Qlbar}$ by weights, which then defines a weight filtration $W_{j}\cohog{i}{\cZ_{\ov s},\Qlbar}$ by the sum of weight spaces of weights $\le j$.  Shrinking $\Spec R$, we have a canonical isomorphism $\cohog{i}{\cZ_{\ov s},\Qlbar}\cong \cohog{i}{Z,\Qlbar}$. Deligne shows that the transport of $W_{\bu}\cohog{i}{\cZ_{\ov s},\Qlbar}$ to $\cohog{i}{Z,\Qlbar}$ defines a filtration $W_{\bu}\cohog{i}{Z,\Qlbar}$ that is independent of the choice of the model $\cZ$ and the closed point $s$. This is the weight filtration on $\cohog{i}{Z}=\cohog{i}{Z,\Qlbar}$. When $k=\CC$, it coincides with the weight filtration from mixed Hodge theory under the canonical isomorphism $\cohog{i}{Z(\CC),\QQ}\ot\Qlbar\cong \cohog{i}{Z,\Qlbar}$.

Define
\begin{equation*}
\cohog{i}{Z}_{\pur}:=\Gr^{W}_{i}\cohog{i}{Z}.
\end{equation*}
When $Z$ is proper over $k$, $\cohog{i}{Z}_{\pur}$ is canonically a quotient of $\cohog{i}{Z}$;  when $Z$ is smooth over $k$, $\cohog{i}{Z}_{\pur}$ is canonically a subspace of $\cohog{i}{Z}$.

For an ind-scheme $X$ over $k$, we use the limit presentation \eqref{coho as lim} to define $W_{j}\cohog{i}{X}$ as $\varprojlim_{Z}W_{j}\cohog{i}{Z}$ as $Z$ runs over all finite-type closed subschemes of $X$.

\section{Construction of perverse filtration}\label{s:perv fil}

This section serves as the technical heart of the paper, although it can be treated as a blackbox. What we will use for the rest of the paper is the existence of a perverse filtration on the pure part of $H_{\g}$, as confirmed in Theorem \ref{th:perv fil}.

The construction of the perverse filtration uses the global analogue of affine Springer fibers, namely parabolic Hitchin fibers. The analogy between affine Springer fibers and Hitchin fibers was the key insight of the work of Laumon--Ng\^o \cite{LN} and Ng\^o \cite{NgoFL} on the Fundamental Lemmas, and its parabolic version was developed in \cite{YGS} and \cite{YLD}.

\subsection{Approximation} 
We show that the cohomology of the affine Springer fiber $\Fl_{\g}$ together with its natural symmetries is unchanged under a perturbation of $\g$.

Let $\frc=\frg\sslash G$. Fix a regular semisimple and integral $\g_{0}\in \frg\lr{t}=\frg(F)$, whose image in $\frc(F)$ is denoted by $a_{0}$. In fact $a_{0}\in\frc(\cO_{F})$ because $\g$ is integral. Let $\k: \frc\to \frg^{\reg}$ be a Kostant section. For $a\in \frc$, denote by $\g(a)\in \frg^{\reg}$ the image of $a$ under the Kostant section $\k$.

Let $J_{a_{0}}$ be the pullback of the regular centralizer group scheme under $a_{0}: \Spec \cO_{F}\to \frc$. Hence $J_{a_{0}}$ is a commutative group scheme over $\cO_{F}$. Note that the centralizer $G_{\g_{0}}$ over $F$ can be identified canonically with the generic fiber of $J_{a_{0}}$; in particular, $LG_{\g_{0}}$ is canonically isomorphic to $LJ_{a_{0}}$. Define the local Picard group over $k$
\begin{equation}\label{local Pa0}
\cP^{\loc}_{a_{0}}:=LJ_{a_{0}}/L^{+}J_{a_{0}}.
\end{equation}
The action of $LG_{\g_{0}}$ on $\Fl_{\g_{0}}$ factors through $\cP^{\loc}_{a_{0}}$.

Recall the following local constancy result on the cohomology of affine Springer fibers.

\begin{lemma}\label{l:approx} There exists a positive integer $N$ (depending on $a_{0}$) such that whenever $a\equiv a_{0}\mod t^{N}$ (as elements of $\frc(k\tl{t}/(t^{N}))$, there is $g\in G(F)$ (depending on $\g_{0}$ and $\g$) such that:
\begin{itemize}
\item Left translation by $g$ on $\Fl_{\g}$ sends $\Fl_{\g_{0}}$ isomorphically to $\Fl_{\g}$. 
\item $\Ad(g)$ induces an isomorphism of group schemes $G_{\g_{0}}\cong G_{\g}$ over $\cO_{F}$.
\end{itemize}
\end{lemma}
\begin{proof}
In \cite[Proposition 3.5.1--Lemme 3.5.4]{NgoFL}, the same statement for the affine Springer fibers in the affine Grassmannian $\Gr$ is proved. That proof can be adapted to our situation with one small change of  \cite[Lemme 3.5.3]{NgoFL}, which should now state: for $h\in G(F)$ and $\g=\g(a)$, we have
\begin{equation}\label{adj Lie J}
\Ad(h^{-1})\g\in \Lie \bI \iff \Ad(h^{-1})\Lie(J_{a})\subset \Lie \bI.
\end{equation}
The direction $\Leftarrow$ is trivial. To see $\Rightarrow$, let $\g'=\Ad(h^{-1})\g\in \Lie \bI$. Then $\Ad(h^{-1})$ induces a map of group schemes $\b: J_{a}\to G_{\g'}$ over $\cO_{F}$. Here $G_{\g'}\subset G_{\cO_{F}}$, and its special fiber is the centralizer $C_{G}(\ov{\g'})$, where $\ov{\g'}=\g'\mod t$, which lies in $\frb$ by assumption. By \cite[Lemma 2.3.1]{YGS}, the special fiber of $\b$ lands in $C_{B}(\ov{\g'})$. In particular, $\Ad(h^{-1})J_{a}\subset \bI$, hence $\Ad(h^{-1})\Lie(J_{a})\subset \Lie \bI$. This proves \eqref{adj Lie J}. The rest of the argument is the same as the proof of \cite[Proposition 3.5.1--Lemme 3.5.4]{NgoFL}.
\end{proof}

\begin{lemma}\label{l:approx equiv}
Let $\g$ and $g$ be as in Lemma \ref{l:approx}. Then $\Ad(g)$ induces an isomorphism of cohomology groups
\begin{equation*}
\Ad(g)^{*}: \cohog{*}{\Fl_{\g}}\cong \cohog{*}{\Fl_{\g_{0}}}
\end{equation*}
satisfying:
\begin{enumerate}
\item $\Ad(g)^{*}$ is equivariant with respect to the $A_{\g}$ and $A_{\g_{0}}$-actions under the canonical isomorphism $A_{\g}\cong A_{\g_{0}}$ (e.g. induced by $\Ad(g^{-1})$, but in fact independent of $g$).
\item $\Ad(g)^{*}$ is $\Wa$-equivariant.
\item The following diagram is commutative 
\begin{equation*}
\xymatrix{ &\cohog{*}{\Fl_{\g}}\ar[dd]^{\Ad(g)^{*}}\\
\cohog{*}{\Fl} \ar[ur]^{i^{*}_{\g}}\ar[dr]_{i^{*}_{\g_{0}}}\\
&\cohog{*}{\Fl_{\g}}}
\end{equation*}
\end{enumerate}
\end{lemma}
\begin{proof}
(1) is clear from the properties of $g$. (3) holds because left translation by $g$ acts by identity on $\cohog{*}{\Fl_{\g}}$ (for $LG$ is connected). For (2), we need to match the actions of each affine simple reflection $s\in \Wa$ under $\Ad(g)^{*}$. Let $\Fl_{\bP_{s}, \g}\subset \Fl_{\bP_{s}}=LG/\bP_{s}$ be the parahoric version of the affine Springer fiber, where $\bP_{s}$ is the standard parahoric subgroup of $LG$ whose Levi quotient $L_{s}$ has roots $\pm\a_{s}$. We have a commutative diagram in which both squares are Cartesian
\begin{equation}\label{two Flg}
\xymatrix{\Fl_{\g_{0}}\ar[r]^-{\Ad(g)}\ar[d]^{\nu^{\g_{0}}_{s}} & \Fl_{\g}\ar[d]^{\nu^{\g}_{s}}\ar[r]^-{\ev_{s,\g}} & [\frb_{s}/B_{s}]\ar[d]^{\nu^{\na}_{s}}\\
\Fl_{\bP_{s}, \g_{0}}\ar[r]^-{\Ad(g)} & \Fl_{\bP_{s}, \g} \ar[r]^-{\ev_{\bP_{s}, \g}} & [\frl_{s}/L_{s}]}
\end{equation}
Here $B_{s}\subset L_{s}$ is the Borel subgroup given by the image of $\bI\subset \bP_{s}$ in $L_{s}$, and $\frb_{s}\subset \frl_{s}$ are the Lie algebras of $B_{s}$ and $L_{s}$.  The map $\nu^{\na}_{s}$ is the Grothendieck alteration for $\frl_{s}$. The map $\ev_{s,\g}$ (resp. $\ev_{\bP_{s},\g}$) sends $g\bI\in \Fl_{\g}$ (resp. $g\bP_{s}\in \Fl_{\bP_{s},\g}$) to the image of $\Ad(g^{-1})\g\in \Lie\bI$ under the projection $\Lie \bI\to \frb_{s}$ (resp. the image of $\Ad(g^{-1})\g\in \Lie\bP_{s}$ under the projection $\Lie \bP_{s}\to \frl_{s}$). The action of $s$ on both $\cohog{*}{\Fl_{\g}}$ and $\cohog{*}{\Fl_{\g_{0}}}$ is induced from the Springer action on $\bR\nu^{\na}_{s*}\Qlbar$ by proper base change. Therefore $\Ad(g)^{*}$ on cohomology is $s$-equivariant. 
\end{proof}

\subsection{Recollections on parabolic Hitchin moduli stack} 
Let $X$ be a smooth projective connected curve over $k$ of genus $g$ and $x\in X$ be a closed point. 

Let $\Bun_{G}$ be the moduli stack of $G$-bundles on $X$. Let $\cL$ be a line bundle over $X$ with $\deg\cL\ge2g$. Let $\cM^{\Hit}=\cM^{\Hit}_{\cL}$ be the moduli stack of $\cL$-twisted $G$-Higgs bundles $(\cE,\ph)$, where $\cE\in \Bun_{G}$ and $\ph$ is a section of $\Ad(\cE)\ot \cL$. Here we denote by $\Ad(\cE)=\cE\times^{G}\frg$ the adjoint bundle of $\cE$. 

Consider the $\Gm$-action on $\frc=\frg\sslash G$ induced from the dilation action on $\frg$. Let $\cA:=\cA_{\cL}$ be the Hitchin base: its points are morphisms $a:X\to \frc/\Gm$ together with an isomorphism between the composition $X\xr{a}\frc/\Gm\to \pt/\Gm$ and the classifying map of the line bundle $\cL$. Choosing homogeneous generators $f_{1},\cdots, f_{r}\in k[\frg]^{G}$ of degree $d_{1},\cdots, d_{r}$, we may identify $\cA$ with $\prod_{i=1}^{r}\cohog{0}{X, \cL^{\ot d_{i}}}$. Let $f^{\Hit}:\cM^{\Hit}\to \cA$ be the Hitchin map sending $(\cE,\ph)$ to the ``characteristic polynomial'' of $\ph$. 

Let $\cA^{\hs}\subset \cA$ be the locus where $a(X)$ generically lies in the open substack $\frc^{\rs}/\Gm$, where $\frc^{\rs}\subset \frc$ is the regular semisimple locus (see \cite[\S4.10.5]{NgoFL}). Following Ng\^o\ \cite[\S4.3]{NgoFL}, there is a Picard stack $\cP$ over $\cA^{\hs}$ whose fiber at $a\in \cA^{\hs}$ is the moduli stack of $J_{a}$-torsors over $X$, where $J_{a}\to X$ is the pullback of the regular centralizer group scheme for $G$ along the map $a: X\to \frc/\Gm$.

As in \cite{YGS}, we consider a parabolic variant of $\cM^{\Hit}$ as follows. Let $\cM:=\cM_{\cL,x}^{\textup{par}}$ be the moduli stack classifying triples $(\cE, \cE^{B}_{x}, \ph)$ where $(\cE,\ph)\in \cM^{\Hit}_{\cL}$ and $\cE^{B}_{x}$ is a $B$-reduction of the fiber $\cE_{x}$ of $\cE$ at $x$, such that the value $\ph(x)\in  \Ad(\cE_{x})\ot \cL_{x}$ lies in $\Ad(\cE^{B}_{x})\ot \cL_{x}$. Here, $\Ad(\cE^{B}_{x})=\cE^{B}_{x}\times^{B}\frb$ is the adjoint bundle of the $B$-bundle $\cE^{B}_{x}$ at the point $x$.   We have the map $\nu: \cM\to \cM^{\Hit}$ forgetting the $B$-reduction at $x$. Let $f=f^{\Hit}\c \nu:\cM\to \cA$ be the composition of $\nu$ with the Hitchin map.

Let $\cM^{\hs}$ be the preimage of $\cA^{\hs}$ under $f$. By \cite[Proposition 2.5.1]{YGS}, $\cM^{\hs}$ is smooth over $k$. By \cite[Lemma 2.3.3]{YGS}, there is an action of $\cP$ on $\cM^{\hs}$ fiberwise over $\cA^{\hs}$.

Let $\cA^{\el}$ be the elliptic locus of $\cA^{\hs}$ (see \cite[\S4.10.5]{NgoFL}), which is the open subset of $\cA^{\hs}$ where fibers of $\cP$ have finitely many connected components. Let $\cM^{\el}$ and $\cP^{\el}$ be the preimage of $\cA^{\el}$ in $\cM$ and $\cP$. By \cite[Corollary 2.5.2]{YGS}, the map $f^{\el}: \cM^{\el}\to \cA^{\el}$ is flat and proper. Moreover, by a theorem of Faltings (see Theorem \ref{th:det ample Hit}), after passing to the coarse moduli space of $\cM^{\el}$, $f^{\el}$ becomes a projective map.

Let $\cK\in D^{b}(\cA^{\el})$ be the direct summand of $\bR f^{\el}_{!}\Qlbar$ on which $\pi_{0}(\cP^{\el}/\cA^{\el})$ acts trivially. Let
\begin{equation*}
{}^{p}\cK^{i}:=\pH^{\dim \cA+i}\cK\in \Perv(\cA).
\end{equation*}
Then ${}^{p}\cK^{i}=0$ unless $0\le i\le 2d$, where $d$ is the relative dimension of $f^{\el}$. The decomposition theorem \cite[Th\'eor\`eme 6.2.5]{BBD} implies a non-canonical decomposition
\begin{equation}\label{non-can decomp}
\cK\cong \bigoplus_{i=0}^{2d}{}^{p}\cK^{i}[-\dim\cA-i].
\end{equation}

The key properties of $\cK$ we will need is

\begin{theorem}[{\cite[Corollary 2.2.4]{YLD}}, following \cite{NgoFL}]\label{th:full supp}
For $0\le i\le 2d$, each simple constituent of ${}^{p}\cK^{i}$ has full support (i.e. equal to $\cA^{\el}$).
\end{theorem}

\subsection{Globalization} 
Denote by $\cO_{x}$, $F_{x}$, and $\fm_{x}$ the completed local ring of $X$ at $x$, its fraction field and its maximal ideal. By choosing a uniformizer of $F_{x}$, we fix a continuous isomorphism of $k$-fields
\begin{equation*}
F_{x}\cong F
\end{equation*}
so that we may view $a_{0}$ as an element of $\frc(\cO_{x})$. We also fix a trivialization of $\cL$ on the formal disk $D_{x}:=\Spec \cO_{x}$ at $x$.

Let $\frD\subset \frc$ be the discriminant divisor. 

\begin{defn} We say $a\in \cA^{\hs}(k)$ is {\em good away from $x$} if $a|_{X-\{x\}}: X-\{x\}\to \frc/\Gm$ is transverse to the divisor $\frD/\Gm$. 
\end{defn}

\begin{lemma}\label{l:approx gamma} Let $N\in \NN$ and suppose $\deg\cL\ge 2g+Nr+1$.  Then there exists $a\in \cA^{\el}(k)$ with the following properties:
\begin{itemize}
\item $a$ is good away from $x$.
\item $a\equiv a_{0}\mod \fm_{x}^{N}$. 
\end{itemize}
\end{lemma}
\begin{proof} We relax the assumption on $G$ to allow it to be reductive.

Let $\cO_{x,N}=\cO_{x}/\fm_{x}^{N}$ and 
\begin{equation*}
\frc_{x,N}=R_{\cO_{x,N}/k}\frc
\end{equation*}
the Weil restriction of $\frc$. It is an affine space of dimension $Nr$. Using the trivialization of $\cL|_{D_{x}}$, we have a map given by the Taylor expansion of $a\in \cA$ at $x$ to the $N$th order
\begin{equation*}
e_{x,N}: \cA\to \frc_{x,N}.
\end{equation*}
For $y\notin x$ we similarly define an evaluation map $e_{y,2}:\cA\to  \frc_{y,2}$. Consider the combined map
\begin{equation*}
e_{(x,N);(y,2)}: \cA\to  \frc_{x,N}\times\frc_{y,2}
\end{equation*}
We claim that $e_{(x,N);(y,2)}$ is surjective and is an affine space fibration if $\deg\cL\ge 2g+N+1$. Indeed, choosing homogeneous generators $f_{1},\cdots, f_{r}\in k[\frg]^{G}$ of degree $d_{1},\cdots, d_{r}$, we identify $\cA$ with $\prod_{i=1}^{r}\cohog{0}{X, \cL^{\ot d_{i}}}$. The map $e_{(x,N);(y,2)}$ is the product over $1\le i\le r$ of the restriction maps
\begin{equation*}
\cohog{0}{X, \cL^{\ot d_{i}}}\to \cL^{\ot d_{i}}\ot_{\cO_{X}}(\cO_{x}/\fm_{x}^{N}\op \cO_{y}/\fm_{y}^{2}).
\end{equation*}
It is surjective if $\cohog{1}{X, \cL^{\ot d_{i}}(-Nx-2y)}=0$, which is guaranteed by our assumption on $\deg\cL$ and $d_{i}\ge1$. This proves that $e_{(x,N);(y,2)}$ is surjective, and its fibers are torsors under $\prod_{i=1}^{r}\cohog{0}{X, \cL^{\ot d_{i}}(-Nx-2y)}$, which is an affine space.

Let $\cA_{a_{0}}=e_{x,N}^{-1}(a_{0})\subset \cA$, an affine subspace of dimension $\dim\cA-\dim\frc_{x,N}=\dim\cA-Nr$. We claim that $\cA^{\el}_{a_{0}}:=\cA_{a_{0}}\cap \cA^{\el}\ne \vn$. Indeed, for a proper Levi subgroup $M$ of $G$, let $\cA_{M}$ be the Hitchin base for $M$ using the same line bundle $\cL$. By \cite[Corollaire 6.3.6]{NgoFL}, $\dim\cA_{M}\le \dim\cA-\deg \cL$, which by our assumption on $\deg\cL$ is less than $\dim\cA-Nr=\dim\cA_{a_{0}}$. Since $\cA^{\el}_{a_{0}}\subset \cA_{a_{0}}$ is the complement $\cA_{a_{0}}\setminus(\cup_{M}\cA_{M})$ where the union is through a finite set of proper Levi subgroups $M$, it is non-empty.

On the other hand, consider the closed subscheme $\frB\subset \cA_{a_{0}}\times (X-\{x\})$ of $(a,y)$ such that $a:X\to \frc/\Gm$ intersects the discriminant divisor $\frD/\Gm$ with multiplicity $\ge2$ at $y$. For fixed $x\ne y\in X(k)$, the fiber $\frB_{y}$ is the preimage in $\cA_{a_{0}}$ of a closed subscheme $\frc^{\textup{bad}}_{y,2}\subset \frc_{y,2}$ under $e_{y,2}$, and $\frc^{\textup{bad}}_{y,2}$ has codimension $\ge2$ in $\frc_{y,2}$. By the smoothness of $e_{(x,N);(y,2)}$, $\dim\frB_{y}\le \dim \cA_{a_{0}}-2$. Therefore, $\dim\frB\le \dim\cA_{a_{0}}-1$. Let $\cA_{a_{0}}^{\textup{bad}}$ be the closure of the image of the projection $\frB\to \cA_{a_{0}}$, then $\dim \cA_{a_{0}}^{\textup{bad}}\le \dim\cA_{a_{0}}-1$. The open complement $\cA_{a_{0}}-\cA_{a_{0}}^{\textup{bad}}$ consists of points that are good away from $x$. 

Putting things together, we conclude that $\cA^{\el}_{a_{0}}\setminus\cA_{a_{0}}^{\textup{bad}}$ is open dense in $\cA_{a_{0}}$, and any $k$-point there satisfies the requirements of the lemma. 
\end{proof}


\subsection{From Hitchin fiber to affine Springer fiber}
For the rest of the section, we fix $N$ as in Lemma \ref{l:approx} and assume $\deg\cL\ge 2g+Nr+1$. Fix also $a\in \cA^{\el}(k)$ as in Lemma \ref{l:approx gamma}, i.e., $a$ is good away from $x$ and is congruent to $a_{0}$ mod $\fm_{x}^{N}$. Let $\g=\g(a)\in\frg(\cO_{x})$ be the image of $a$ under the Kostant section. By Lemma \ref{l:approx equiv}, we have an isomorphism
\begin{equation}\label{isom coho Flg}
\cohog{*}{\Fl_{\g}}\cong \cohog{*}{\Fl_{\g_{0}}}
\end{equation}
equivariant under the actions of $A_{\g}$ (resp. $A_{\g_{0}}$), $\Wa$ and $\cohog{*}{\Fl}$.

The commutative group scheme $J_{a}$ over $X$ is generically a torus. Let $J^{\flat}_{a}$ be the finite type N\'eron model of the generic fiber of $J_{a}$. The quotient $J^{\flat}_{a}/J_{a}$ is supported at finitely many closed points of $X$; write $R_{a,y}$ for its fiber at $y\in X(k)$. Since $a$ is good away from $x$, for $y\ne x$, $R_{a,y}$ is a finite group scheme over $k$. Therefore $R_{a}^{x}:=\prod_{y\ne x}R_{a,y}$ is a finite group scheme over $k$. Define
\begin{equation*}
J'=\ker(J^{\flat}_{a}\to \pi_{0}(R_{a,x})\times R^{x}_{a}).
\end{equation*}
By definition we have a short exact sequence $1\to J_{a}\to J'_{a}\to i_{x*}R^{\c}_{a,x}\to 1$ on $X$. Taking the Picard stack of torsors we get an exact sequence of Picard stacks
\begin{equation}\label{global Pic seq}
1\to R^{\c}_{a,x}\to \cP_{a}\to \cP'_{a}\to 1
\end{equation}
where $\cP'_{a}$ is the moduli stack of $J'_{a}$-torsors over $X$. By \cite[Proposition 4.8.2(2)]{NgoFL}, the neutral component of the moduli stack $\cP^{\flat}_{a}$ of $J^{\flat}_{a}$ is an abelian stack (an abelian variety quotient by a finite diagonalizable group). Since $\cP'_{a}\to \cP^{\flat}_{a}$ is a torsor for the finite group scheme $\pi_{0}(R_{a,x})\times R_{a}^{x}$, the neutral component of $\cP'_{a}$ is also an abelian stack. We denote its coarse moduli space by $A_{a}$, an abelian variety over $k$.

The local Picard stack
\begin{equation*}
\cP_{a,x}:=L_{x}J_{a}/L^{+}_{x}J_{a}.
\end{equation*}
can be identified with $\cP^{\loc}_{a_{x}}$ introduced in \eqref{local Pa0}.  
Here $L_{x}J_{a}$ is the loop group of $J_{a}$ using the Laurent series field $F_{x}$. Similarly let
$
\cP'_{a,x}:=L_{x}J_{a}/L^{+}_{x}J'_{a}.
$
We have an exact sequence
\begin{equation}\label{local Pic seq}
1\to R^{\c}_{a,x}\to \cP_{a,x}\to \cP'_{a,x}\to1. 
\end{equation}
Denote by $\L_{a,x}$ the reduced structure of $\cP'_{a,x}$, which is a finitely generated abelian group whose torsion-free quotient has rank equal to the split rank of $J_{a}|_{\Spec F_{x}}\cong G_{\g}$. Let $P_{a,x}\subset \cP_{a,x}$ be the preimage of $\L_{a,x}$, which has the same reduced structure as $\cP_{a,x}$. We thus have an exact sequence
\begin{equation*}
1\to R^{\c}_{a,x}\to P_{a,x}\to \L_{a,x}\to1. 
\end{equation*}
The exact sequences \eqref{local Pic seq} and \eqref{global Pic seq} fit into a commutative diagram of exact sequences
\begin{equation}\label{local to global Pic seq}
\xymatrix{1 \ar[r] &  R^{\c}_{a,x}\ar@{=}[d]\ar[r] &  P_{a,x}\ar[d]\ar[r] &  \L_{a,x}\ar[d]\ar[r] & 1 \\
1\ar[r] & R^{\c}_{a,x}\ar[r] & \cP_{a}\ar[r] & \cP'_{a}\ar[r] & 1}
\end{equation}
where the vertical maps are the local to global canonical maps.

\begin{lemma} Let $\cM_{a}=f^{-1}(a)$. There is a canonical homeomorphism
\begin{equation}\label{prod form}
\cP_{a}\times^{\cP_{a,x}}\Fl_{\g}\to \cM_{a}.
\end{equation}
\end{lemma}
\begin{proof} For $y\in X(k)$, after choosing a trivialization of $\cL|_{\Spec \cO_{y}}$, we may identify $a|_{\Spec \cO_{y}}$ with an element of $\frc(\cO_{y})$. Let $\g_{y}(a)\in \frg(\cO_{y})$ be the corresponding point in the Kostant section.  The product formula \cite[Proposition 2.4.1]{YGS} (following \cite[Proposition 4.15.1]{NgoFL}) states that there is a canonical homeomorphism
\begin{equation*}
\cP_{a}\times^{\cP_{a,x}\times \prod_{y\ne x}\cP_{a,y}}(\Fl_{\g}\times \prod_{y\ne x} \Gr_{\g_{y}(a)})\to \cM_{a}.
\end{equation*}
The product is over the finitely many points where $R_{a,y}$ is nontrivial. Since $R_{a,y}$ is finite for $y\ne x$, $\Gr_{\g_{y}(a)}$ is $0$-dimensional, hence a $\cP_{a,y}$-torsor, which allows us to simplify the above formula to the desired form \eqref{prod form}.
\end{proof}

\subsection{Comparing cohomology} 


Choose a section to the surjection $P_{a,x}\to \L_{a,x}$ \footnote{Our assumption on $\ch(k)$ ensures that $\ch(k)$ is prime to $|W|$, hence the torsion part of $\L_{a,x}$, which is isomorphic to the torsion part of $\xcoch(T)_{w}$ for some $w\in W$, is prime to $p$. This implies that a section to $P_{a,x}\to \L_{a,x}$ exists since the kernel $R^{\c}_{a,x}(k)$ is $n$-divisible for any $n$ prime to $p$.} and call the image $\L_{\g}\subset P_{a,x}$. In view of \eqref{local to global Pic seq}, \eqref{prod form} implies a $\cP_{a}$-equivariant homeomorphism
\begin{equation}\label{Fl to Ma}
M_{a}:=\cP'_{a}\times^{\L_{\g}}\Fl_{\g}\cong \cP_{a}\times^{P_{a,x}}\Fl_{\g}\to\cP_{a}\times^{\cP_{a,x}}\Fl_{\g}\to \cM_{a}.
\end{equation}

Denote by $\Gr^{W}_{i}(-)$ the associated graded with respect to the weight filtration, see \S\ref{ss:wt}. Since $\cM_{a}$ is proper, $\Gr^{W}_{i}\cohog{i}{\cM_{a}}$ is the largest weight quotient of $\cohog{i}{\cM_{a}}$. Note that the stalk $\cK_{a}$ is $\cohog{*}{\cM_{a}}^{\pi_{0}(\cP_{a})}$. Let
\begin{equation*}
\cK_{a,\pur}=\bigoplus_{i\in\ZZ}(\Gr^{W}_{i}\cohog{i}{\cM_{a}})^{\pi_{0}(\cP_{a})}
\end{equation*}
which is canonically a quotient of $\cK_{a}$.  Similarly define
\begin{equation*}
\cohog{*}{\L_{\g}\bs \Fl_{\g}}_{\pur}=\bigoplus_{i\in\ZZ}\Gr^{W}_{i}\cohog{i}{\L_\g\bs \Fl_\g}.
\end{equation*}

\begin{lemma}\label{l:pure part invt}
Pullback along the projection $\Fl_{\g}\to\L_\g\bs \Fl_\g$ induces a graded isomorphism
\begin{equation*}
\cohog{*}{\L_\g\bs \Fl_\g}_{\pur}\isom \cohog{*}{\Fl_{\g}}^{\L_{\g}}_{\pur}=:H_{\g,\pur}.
\end{equation*}
\end{lemma}
\begin{proof}
There is a spectral sequence with $E_{2}$-page $E_{2}^{i,j}=\cohog{i}{\L_{\g}, \cohog{j}{\Fl_{\g}}}$ convergent to $\cohog{i+j}{\L_\g\bs \Fl_\g}$. The differentials are strictly compatible with the weight filtrations. Passing to the pure part, namely replacing $E_{2}^{i,j}$ by $\Gr^{W}_{i+j}E_{2}^{i,j}$, we have $\Gr^{W}_{i+j}E_{2}^{i,j}=0$ unless $i=0$, in which case $\Gr^{W}_{j}E_{2}^{0,j}=\cohog{j}{\Fl_{\g}}_{\pur}^{\L_{\g}}$. Therefore the pure part of the spectral sequence degenerates at $E_{2}$ with the only nonzero entries $\cohog{j}{\Fl_{\g}}_{\pur}^{\L_{\g}}$ at $(0,j)$. This implies that the pullback along $\Fl_{\g}\to\L_\g\bs \Fl_\g$ induces a canonical isomorphism $\cohog{j}{\L_\g\bs \Fl_\g}_{\pur}\isom \cohog{j}{\Fl_{\g}}^{\L_{\g}}_{\pur}$ for all $j\in \ZZ$.
\end{proof}

Recall that the abelian variety $A_{a}$ is the coarse moduli space of the neutral component of $\cP'_{a}$. We have the map
\begin{equation}\label{prod to Ma}
\a: A_{a}\times \Fl_{\g}\to \cP'_{a}\times \Fl_{\g}\to M_{a}
\end{equation}

\begin{lemma}\label{l:Ka tensor prod}
Pullback along $\a$ induces an isomorphism of graded rings
\begin{equation*}
\a^{*}: \cK_{a,\pur}\isom \cohog{*}{A_{a}}\ot H_{\g,\pur}
\end{equation*}
\end{lemma}
\begin{proof}
The map $\b: \cP'_{a}\times \Fl_{\g}\to M_{a}$ is a $\L_{\g}$-torsor. The argument of Lemma \ref{l:pure part invt} shows that $\b^{*}$ restricts to an isomorphism
\begin{equation*}
\cohog{*}{M_{a}}_{\pur}\isom (\cohog{*}{\cP'_{a}}\ot \cohog{*}{\Fl_{\g}}_{\pur})^{\L_{\g}}.
\end{equation*}
Here the action of $\L_{\g}$ is diagonal on both factors. Taking invariants under $\pi_{0}(\cP_{a})=\pi_{0}(\cP'_{a})$, we get
\begin{equation*}
\cK_{a,\pur}=\cohog{*}{M_{a}}^{\pi_{0}(\cP_{a})}_{\pur}\isom\cohog{*}{\cP'_{a}}^{\pi_{0}(\cP'_{a})}\ot \cohog{*}{\Fl_{\g}}_{\pur}^{\L_{\g}}\cong \cohog{*}{A_{a}}\ot H_{\g,\pur}.
\end{equation*}
\end{proof}

\subsection{Cap product}
By the homeomorphism
\eqref{Fl to Ma}, the action of $\cP'_{a}$ on $M_{a}=\cP'_{a}\times^{\L_{\g}}\Fl_{\g}$ induces a graded action of the homology $\homog{*}{A_{a}}$ (viewed as a graded algebra) on $\cK_{a}\cong\cohog{*}{M_{a}}^{\pi_{0}(\cP'_{a})}$. Since $\homog{*}{A_{a}}$ is pure, the action induces an action of $\homog{*}{A_{a}}$ on $\cK_{a,\pur}$. Under the isomorphism in Lemma \ref{l:Ka tensor prod}, the action of $\homog{*}{A_{a}}$ on $\cK_{a,\pur}$ corresponds to the cap product action on the $\cohog{*}{A_{a}}$-factor, and identity on the $H_{\g,\pur}$-factor.

Consider the $\cP'_{a}$-torsor
\begin{equation*}
\pi: M_{a}=\cP_{a}\times^{\L_{\g}}\Fl_{\g}\to \L_\g\bs \Fl_\g
\end{equation*}
The pullback along $\pi$ induces a graded map
\begin{equation*}
\pi^{*}: \cohog{*}{\L_\g\bs \Fl_\g}\to \cohog{*}{M_{a}}^{\pi_{0}(\cP'_{a})}=\cK_{a}.
\end{equation*}
Restricting to the pure parts and using Lemma \ref{l:pure part invt} we get
\begin{equation}\label{pi pb pure}
\pi^{*}_{\pur}: H_{\g,\pur}\to \cK_{a,\pur}.
\end{equation}
This map identifies with the inclusion of the summand $\cohog{0}{A_{a}}\ot H_{\g,\pur}$ under the isomorphism in Lemma \ref{l:Ka tensor prod}.

On the other hand, Gysin pushforward along $\pi$ gives a canonical map
\begin{equation*}
\tau: \cK_{a}\cong\cohog{*}{M_{a}}^{\pi_{0}(\cP'_{a})}\to \cohog{2\ov d}{A_{a}}\ot \cohog{*}{\L_{\g}\bs \Fl_{\g}}.
\end{equation*}
Passing to the pure parts, we get a canonical map
\begin{equation}\label{proj Ka to Hg}
\t_{\pur}: \cK_{a,\pur}\to H_{\g,\pur}[-2\ov d](-\ov d)\cong \cohog{2\ov d}{A_{a}}\ot H_{\g,\pur}.
\end{equation}
This map identifies with the projection to the summand $\cohog{2d}{A_{a}}\ot H_{\g,\pur}$ under the isomorphism in Lemma \ref{l:Ka tensor prod}.

Recall that in \cite[\S5]{YGS} we have defined an action of $\Wa$ on the complex $\cK$, hence on $\cK_{a}$ and $\cK_{a,\pur}$.
\begin{lemma}\label{l:Wa equiv}
The maps $\pi^{*}$ and $\t$ are $\Wa$-equivariant.
\end{lemma}
\begin{proof} We need to show that $\t$ and $\pi^{*}$ are equivariant for each affine simple reflection $s\in \Wa$. Let $\bP_{s}\supset \bI$ be the standard parahoric subgroup of $LG$ whose reductive quotient $L_{s}$ has roots $\pm\a_{s}$. We have the parahoric analog $\cM_{\bP_{s}}$ of $\cM$ (simply changing the $B$-reduction at $x$ to a $\bP_{s}$-level structure) and its fiber $\cM_{\bP_{s},a}$ over $a\in \cA(k)$. See \cite[\S4.3]{YGS} for more details. Similarly we have the local analog $\Fl_{\bP_{s},\g}$, the affine Springer fiber in the affine partial flag variety $\Fl_{\bP_{s}}=LG/\bP_{s}$. The product formula has a parahoric analog, a homeomorphism $M_{\bP_{s},a}:=\cP'_{a}\times^{\L_{\g}}\Fl_{\bP_{s},\g}\to \cM_{\bP_{s}, a}$. We have a commutative diagram in which both squares are Cartesian
\begin{equation*}
\xymatrix{M_{a}\ar[r]^-{\pi}\ar[d]^{\nu_{s}^{a}} & \L_{\g}\bs \Fl_{\g}\ar[d]^{\nu^{\g}_{s}}\ar[r]^-{\ev_{s}^{\g}} & [\frb_{s}/B_{s}]\ar[d]^{\nu^{\na}_{s}}\\
M_{\bP_{s}, a}\ar[r]^-{\pi_{s}} & \L_{\g}\bs \Fl_{\bP_{s}, \g} \ar[r]^-{\ev_{\bP_{s}, \g}} & [\frl_{s}/L_{s}]}
\end{equation*}
The right half of the diagram is the same as that of \eqref{two Flg}. The Springer action of $s$ on $\bR\nu^{\na}_{s*}\Qlbar$ induces an action on $\bR\nu^{\g}_{s*}\Qlbar$  and $\bR\nu^{a}_{s*}\Qlbar$ by proper base change. By construction, the $s$-actions on $\cohog{*}{\cM_{a}}\cong \cohog{*}{M_{a}}$ and on $\cohog{*}{\L_{\g}\bs \Fl_{\g}}$ are induced from the $s$-actions on $\bR\nu^{\g}_{s*}\Qlbar$  and $\bR\nu^{a}_{s*}\Qlbar$. We conclude that both the pullback $\pi^{*}$ and the Gysin map $\t$ for the projection $\pi: M_{a}\to \L_{\g}\bs \Fl_{\g}$ is  $s$-equivariant.
\end{proof}

\subsection{Global and local perverse filtrations} 
For $i\in \ZZ$ let
\begin{equation*}
P_{i}\cK_{a}:=({}^{p}\t_{\le \dim \cA+i}\cK)_{a}.
\end{equation*}
By the non-canonical decomposition \eqref{non-can decomp}, the natural map $P_{i}\cK_{a}\to \cK_{a}$ is injective in each cohomological degree, giving an increasing filtration $P_{i}\cK_{a}$ on $\cK_{a}$. Moreover, $\Gr^{P}_{i}\cK_{a}=0$ unless $0\le i\le 2d$. Let $P_{i}\cK_{a,\pur}$ be the image of $P_{i}\cK_{a}$ in $\cK_{a,\pur}$. In \cite[\S7.7.2]{NgoFL} it is explained that there is a canonical action of $\homog{*}{A_{a}}$ on the perverse associated graded $\Gr^{P}_{*}\cK_{a}$. Indeed, $\homog{j}{A_{a}}$ sends $\Gr^{P}_{i}\cK_{a}$ to $\Gr^{P}_{i-j}\cK_{a}$ for $i\in \ZZ, j\ge0$. Moreover, the following is true.

\begin{prop}[{\cite[Proposition 7.7.4]{NgoFL}}]\label{p:perv Gr free} As a graded $\homog{*}{A_{a}}$-module, $\Gr^{P}_{*}\cK_{a}$ is free (under the perverse grading). In particular, passing to the pure part, $\Gr^{P}_{*}\cK_{a,\pur}$ is a free graded $\homog{*}{A_{a}}$-module.
\end{prop}

\begin{lemma}\label{l:cup any}
For any cohomology class $\xi\in \cohog{n}{\cM_{a}}$, the cup product with $\xi$ sends $P_{i}\cK_{a}$ to $(P_{i+n}\cK_{a})[n]$, for all $i\in \ZZ$. In particular, the associated graded $\Gr^{P}_{*}\cK_{a}$ (with the perverse grading) is a graded $\cohog{*}{\cM_{a}}$-module under the cup product. 
\end{lemma}
\begin{proof}
Replace the Hitchin base $\cA$ by the strict henselization $S$ of $\cA$ at $a$, and let $\cM_{S}=S\times_{\cA}\cM$, with special fiber $\cM_{a}$. By proper base change, the restriction $\cohog{*}{\cM_{S}}\to \cohog{*}{\cM_{a}}$ is an isomorphism, therefore $\xi$ extends to a class $\wt\xi\in \cohog{n}{\cM_{S}}$. Let $\cK_{S}$ be the pullback of $\cK$ to $S$. Then $\cup\xi: \cK_{S}\to \cK_{S}[n]$ sends ${}^{p}\t_{\le i}\cK_{S}$ to $({}^{p}\t_{\le i+n}\cK_{S})[n]$. Passing to stalks at $a$, we conclude that $\cup\xi$ sends $P_{i}\cK_{a}$ to $(P_{i+n}\cK_{a})[n]$.
\end{proof}

We now define a filtration $P_{i}$ on $H_{\g,\pur}=\cohog{*}{\Fl_{\g}}^{\L_{\g}}_{\pur}$. 

\begin{defn}
For $i\in \ZZ$, let $P_{i}H_{\g,\pur}\subset H_{\g,\pur}$ be the graded subspace such that  $P_{i}H_{\g,\pur}[-2\ov d](-\ov d)$ is the image of $P_{i+2\ov d}\cK_{a,\pur}$ under $\t_{\pur}$. 
\end{defn}


\begin{lemma}\label{l:cup any local}
For any cohomology class $\xi\in H^{n}_{\g,\pur}$, the cup product with $\xi$ sends $P_{i}H_{\g,\pur}$ to $(P_{i+n}H_{\g,\pur})[n]$, for all $i\in \ZZ$. In particular, the associated graded $\Gr^{P}_{*}H_{\g,\pur}$ (with the perverse grading) is a graded $H_{\g,\pur}$-module under the cup product. 
\end{lemma}
\begin{proof}
The class $\xi$ can be viewed as the class $1\ot \xi\in \cohog{0}{A_{a}}\ot H^{n}_{\g,\pur}\subset \cK_{a,\pur}$ under the isomorphism $\a^{*}$ in Lemma \ref{l:Ka tensor prod}. By Lemma \ref{l:cup any}, $\cup \xi$ sends $P_{2\ov d+i}\cK_{a,\pur}$ to $(P_{2\ov d+i+n}\cK_{a,\pur})[n]$. Taking images under $\t_{\pur}$,  which intertwines the cup product action of $\xi$, we conclude that $\cup\xi$ sends $P_{i}H_{\g,\pur}$ to $(P_{i+n}H_{\g,\pur})[n]$. 
\end{proof}

The map $\t_{\pur}$ induces a canonical surjection of doubly graded vector spaces
\begin{equation}\label{GrP tau}
\Gr^{P}_{*}\t_{\pur}: \Gr^{P}_{*}\cK_{a,\pur}\to (\Gr^{P}_{*-2\ov d}H_{\g,\pur})[-2\ov d](-\ov d).
\end{equation}

\begin{lemma}\label{l:GrP global reduces to GrP local}
The canonical map $\Gr^{P}_{*}\t_{\pur}$ induces a bigraded isomorphism
\begin{equation}\label{GrP isom}
(\Gr^{P}_{*}\cK_{a,\pur})\ot_{\homog{*}{A_{a}}}\Qlbar\isom (\Gr^{P}_{*-2\ov d}H_{\g,\pur})[-2\ov d](-\ov d).
\end{equation}
In particular, if $\Gr^{P}_{i}H_{\g,\pur}\ne 0$, then $0\le i\le 2d_{\g}$.
\end{lemma}
\begin{proof}
From the construction we have a surjection of cohomologically graded vector spaces $\Gr^{P}_{i+2\ov d}\cK_{a,\pur}\surj\Gr^{P}_{i}H_{\g,\pur}[-2\ov d](-\ov d)$ for all $i$. Therefore \eqref{GrP isom} in question is surjective. To show \eqref{GrP isom} is an isomorphism, it suffices to compare the total dimension. By  the freeness result in \eqref{p:perv Gr free}, the left side of \eqref{GrP isom} has dimension equal to $\dim\Gr^{P}_{*}\cK_{a,\pur}/\dim \homog{*}{A_{a}}$, which is the same as $\dim  \cK_{a,\pur}/\dim \homog{*}{A_{a}}$. By Lemma \ref{l:Ka tensor prod}, $\dim  \cK_{a,\pur}/\dim \homog{*}{A_{a}}=\dim H_{\g,\pur}$, which is the same as $\dim \Gr^{P}_{*}H_{\g,\pur}$.  Thus the two sides of \eqref{GrP isom} have the same total dimension. The lemma is proved.
\end{proof}

\begin{prop}\label{p:GrP tensor prod}
Recall the graded ring isomorphism $\a^{*}: \cK_{a,\pur}\isom \cohog{*}{A_{a}}\ot H_{\g,\pur}$ from Lemma \ref{l:Ka tensor prod}.
\begin{enumerate}
\item The perverse filtration on $\cK_{a,\pur}$ is compatible with the grading of $\cK_{a,\pur}$ by the cohomological degree from the $\cohog{*}{A}$-factor. More precisely, for each $j\in \ZZ$, 
\begin{equation*}
P_{j}\cK_{a,\pur}=\bigoplus_{i=0}^{2\ov d}\cohog{i}{A_{a}}\ot P_{j-i}H_{\g,\pur}.
\end{equation*}
\item The map $\a^{*}$ induces a canonical doubly-graded (cohomological and perverse) isomorphism of $\homog{*}{A_{a}}$-modules 
\begin{equation}\label{GrP prod}
\Gr^{P}_{*}\a^{*}: \Gr^{P}_{*}\cK_{a,\pur}\isom \cohog{*}{A_{a}}\ot \Gr^{P}_{*}H_{\g,\pur}.
\end{equation}
Here the action of $\homog{*}{A_{a}}$ on the right side is via cap product on the first factor.

\item The isomorphism \eqref{GrP prod} is compatible with the cup product actions of $\cK_{a,\pur}$ on the left and $\cohog{*}{A_{a}}\ot H_{\g,\pur}$ on the right via $\a^{*}$ (see Lemma \ref{l:cup any} and Lemma \ref{l:cup any local}).
\end{enumerate}
\end{prop}
\begin{proof}
(1) We first show that
\begin{equation}\label{HiA Pj}
\cohog{i}{A_{a}}\ot P_{j-i}H_{\g,\pur}\subset P_{j}\cK_{a,\pur}.
\end{equation}
By Lemma \ref{l:cup any}, cup product with $\cohog{i}{A_{a}}\subset \cohog{i}{\cM_{a}}$ sends $P_{j-i}\cK_{\g,\pur}$ to $(P_{j}\cK_{a,\pur})[i]$, therefore it suffices to prove the case $i=0$, i.e., 
\begin{equation}\label{1Pj}
1\ot P_{j}H_{\g,\pur}\subset P_{j}\cK_{a,\pur}.
\end{equation}
Let $\xi\in P_{j}H_{\g,\pur}$ and $\wt\xi\in P_{2\ov d+j}H_{\g,\pur}$ be any preimage of $\xi$ under $\t_{\pur}$. Then $\int_{A_{a}}\wt\xi\in  P_{j}\cK_{a,\pur}$. Moreover, under the isomorphism $\a$, $\int_{A_{a}}\wt\xi=1\ot \xi\in 1\ot \ot P_{j}H_{\g,\pur}$. Therefore, \eqref{1Pj} holds and \eqref{HiA Pj} is proved. To show the inclusion in \eqref{HiA Pj} is an equality, it suffices to compare dimensions of both sides, which is provided by Lemma \ref{l:GrP global reduces to GrP local}. 

(2) follows from (1) by taking associated graded of the perverse filtration.

(3) follows from the fact that $\a^{*}$ is a ring isomorphism. 
\end{proof}

\begin{remark} By Proposition \ref{p:GrP tensor prod}(1), $P_{j}H_{\g,\pur}$ may be characterized alternatively as the preimage of $P_{j}\cK_{a,\pur}$ under the pullback map $\pi_{\pur}^{*}: H_{\g, \pur}\to \cK_{a,\pur}$ in \eqref{pi pb pure}. 
\end{remark}


\subsection{Relative hard Lefschetz}
The goal of this subsection is to verify that the filtration $P_{i}H_{\g,\pur}$ satisfies property \eqref{HL} of Definition \ref{d:perv fil}.

Choose a faithful representation $V$ of $G$ of dimension $n$, we get a map $\io_{V}: \Bun_{G}\to \Bun_{n}$ given by $\cE\mt \cE\times^{G}V$, the latter being the moduli stack of rank $n$ vector bundles on $X$. Let $\LL_{\det}\in \Pic(\Bun_{G})$ be the determinant line bundle: its fiber at $\cV\in \Bun_{n}$ is the determinant of cohomology $\det \bR\G(X,\cV)$. Let $\LL_{V}=\io_{V}^{*}\LL_{\det}\in \Pic(\Bun_{G})$.

We have a local analog of $\LL_{\det}$ and $\LL_{V}$. Let $\Gr_{n}$ be the affine Grassmannian of $\GL_{n}$, which is the moduli space of rank $n$ $k\tl{t}$-submodules $\L$ of $k\lr{t}^{n}$. Let $\LL_{\det,\Gr}$ be the line bundle on $\Gr_{n}$ whose fiber at $\L\subset k\lr{t}^{n}$ is
\begin{equation*}
\det(\L/t^{N}\L_{0})\ot \det(\L_{0}/t^{N}\L_{0})^{-1}
\end{equation*}
where $\L_{0}=k\tl{t}^{n}$ and $N$ is sufficiently large so that $t^{N}\L_{0}\subset \L$, and $\det(-)$ means top exterior power of a $k$-vector space. Let $\io_{V,\Gr}: \Gr_{G}\to \Gr_{n}$ be the map induced by the representation $G\to \GL(V)$, and let $\LL_{V,\Gr}=\io_{V,\Gr}^{*}\LL_{\det, \Gr}$.

Below we fix $V$ and omit it from notations, so we have the line bundles $\LL=\LL_{V}$ on $\Bun_{G}$ and $\LL_{\Gr}=\LL_{V,\Gr}$ on $\Gr_{G}$, etc. We have forgetful maps
\begin{equation*}
\cM\to \cM^{\Hit}\to \Bun_{G}
\end{equation*}
We denote  by $\LL_{\cM}$ and $\LL_{\Hit}$ the pullbacks of $\LL$ to $\cM$ and $\cM^{\Hit}$.


\begin{lemma}\label{l:det line factor}
Let $\LL_{\cM_{a}}$ be the restriction of $\LL_{\cM}$ to $\cM_{a}$. Consider the map $\b: \cP_{a}\times \Fl_{\g}\to \cM_{a}$ from the product formula \eqref{prod form}. Then
\begin{equation*}
\b^{*}\LL_{\cM_{a}}\cong \LL_{\cP_{a}}\bt (p_{\g}^{*}\LL_{\Gr}).
\end{equation*}
Here $\LL_{\cP_{a}}$ is a line bundle on $\cP_{a}$, and $p_{\g}: \Fl_{\g}\to \Gr$ is the projection.
\end{lemma}
\begin{proof} Let $\cM_{a}^{\Hit}$ be the Hitchin fiber over $a$, and $\b': \cP_{a}\times \Gr_{\g}\to \cM_{a}^{\Hit}$ be the map defined in \cite[\S4.15]{NgoFL}, which is the analogue of $\b$ without Iwahori level structures at $x$. Let $\om_{a}: \cM_{a}\to \Bun_{G}$ be map forgetting the Higgs field. It suffices to show that
\begin{equation}\label{det line factor Hit}
\b'^{*}\om_{a}^{*}\LL\cong \LL_{\cP_{a}}\bt (\LL_{\Gr}|_{\Gr_{\g}})
\end{equation}
Consider the map $s_{\g}: \Gr_{\g}\to \cM_{a}$ obtained by restricting $\a'$ to $\{1\}\times \Gr_{\g}$. This is obtained by modifying a fixed $G$-Higgs bundle $(\cE_{0},\ph_{0})\in \cM^{\Hit}_{a}$ (the Hitchin-Kostant section over $a$) at $x$. Let $\Hk_{x}$ be the Hecke correspondence at $x$ that classifies $(\cE,\cE', \t)$ where $\cE,\cE'$ are $G$-bundles on $X$ and $\t$ is an isomorphism $\cE|_{X-\{x\}}\cong \cE'|_{X-\{x\}}$ of $G$-bundles. We have maps 
\begin{equation*}
\xymatrix{&\Hk_{x}\ar[rr]^-{\ev_{x}}\ar[dl]_{h}\ar[dr]^{h'} && \Hk^{\loc}_{x}:=L^{+}G\bs LG/L^{+}G\\
\Bun_{G} & & \Bun_{G}}
\end{equation*}
Upon choosing a trivialization of $\cE_{0}|_{\Spec \cO_{x}}$, we identify $\Gr_{G}$ with the fiber $h^{-1}(\cE_{0})$, so that $\ev_{x}|_{\Gr_{G}}$ becomes the canonical projection $q: \Gr_{G}\to \Hk^{\loc}_{x}=L^{+}G\bs \Gr_{G}$.  Restricting to $\Gr_{\g}$ we get a map $\Gr_{\g}\incl \Gr_{G}\cong h^{-1}(\cE_{0})\incl \Hk_{x}$. For $\xi\in \Gr_{\g}$, denoting its image in $\Hk_{x}$ by $(\cE_{0},\cE_{\xi}, \t)$. By construction, $\cE_{0}$ and $\cE_{\xi}$ carry Higgs fields $\ph_{0}$ and $\ph_{\xi}$ that are identified under $\t$ over $X-\{x\}$, and $J_{a}$ acts on both $(\cE_{0},\ph_{0})$ and $(\cE_{\xi}, \ph_{\xi})$. For any $\cP_{a}$-torsor $\cQ$, $\b$ sends $(\cQ, \xi)$ to $\cQ\times^{J_{a}}(\cE_{\xi}, \ph_{\xi})$. We define a map
\begin{equation*}
\wt\b: \cP_{a}\times \Gr_{\g}\to \Hk_{x}
\end{equation*}
sending $(\cQ,\xi)$ to $(\cQ\times^{J_{a}}\cE_{0}, \cQ\times^{J_{a}}\cE_{\xi}, \t')$ where $\t'$ is induced from $\t$ and the identity map on $\cQ|_{X-\{x\}}$. Let $\mu_{a}:\cP_{a} \to \Bun_{G}$ be the map sending $\cQ$ to $\cQ\times^{J_{a}}\cE_{0}$. Then we have
\begin{eqnarray}\label{hb}
&&h\c \wt\b=\mu_{a}\c \pr_{\cP_{a}}: \cP_{a}\times \Gr_{\g}\to \Bun_{G},\\
\label{evb}
&&\ev\c\wt\b=q_{\g}\c\pr_{\Gr_{\g}}: \cP_{a}\times \Gr_{\g}\to \Hk^{\loc}_{x}, \quad \mbox{where $q_{\g}=q|_{\Gr_{\g}}: \Gr_{\g}\to \Hk^{\loc}_{x}$,}\\
\label{h'b}
&&h'\c\wt\b=\om_{a}\c\b: \cP_{a}\times \Gr_{\g}\to \Bun_{G}.
\end{eqnarray}
From the definition of determinant line bundles, we have a canonical isomorphism of line bundles on $\Hk_{x}$
\begin{equation}\label{line bd on Hk}
h'^{*}\LL\cong h^{*}\LL\ot \ev^{*}\LL_{x}.
\end{equation}
Here $\LL_{x}$ is the line bundle on $\Hk^{\loc}_{x}$ that is the canonical descent of $\LL_{\Gr}$ on $\Gr_{G}$. Pulling back \eqref{line bd on Hk} by $\wt\b$, and using the identities \eqref{h'b}, \eqref{hb} and \eqref{evb} and , we get
\begin{equation*}
\b'^{*}\om_{a}^{*}\LL=\wt\b^{*}h'^{*}\LL\cong \wt\b^{*}h^{*}\LL\ot \wt\b^{*}\ev^{*}\LL_{x}\cong \pr^{*}_{\cP_{a}}\mu_{a}^{*}\LL\ot \pr_{\Gr_{\g}}^{*}q_{\g}^{*}\LL_{x}=\mu_{a}^{*}\LL\bt (\LL_{\Gr}|_{\Gr_{\g}}).
\end{equation*}
We get \eqref{det line factor Hit} by setting $\LL_{\cP_{a}}=\mu_{a}^{*}\LL$.
\end{proof}

We recall the following result due to Faltings.
\begin{theorem}[{\cite[Theorem II.5]{F}}]\label{th:det ample Hit} The restriction of $\LL_{\Hit}$ to the coarse moduli space of $\cM^{\Hit,\el}=\cM^{\Hit}|_{\cA^{\el}}$ is anti-ample.
\end{theorem}
Indeed, Faltings showed that $\LL_{\Hit}$ is anti-ample when restricted to the stable locus of $\cM^{\Hit}$, which contains $\cM^{\Hit,\el}$.

\begin{lemma}\label{l:global HL} Let $\y_{\cM}=c_{1}(\LL_{\cM})\in \cohog{2}{\cM}(1)$. Then the complex $\cK$ satisfies relative hard Lefschetz with respect to $\y_{\cM}$. In other words, cupping with $\y_{\cM}$ sends ${}^{p}\t_{\le i}\cK$ to $({}^{p}\t_{\le i+2}\cK)[2]$; and for $0\le i\le d$, cupping with $\y_{\cM}^{i}$ induces an isomorphism in $\Perv(\cA^{\el})$
\begin{equation}\label{cup eta i global}
\cup \y_{\cM}^{i}: {}^{p}\cK^{d-i}\isom {}^{p}\cK^{d+i}[2i](i).
\end{equation}
In particular, taking stalk at $a$ and letting $\y_{\cM_{a}}=c_{1}(\LL_{\cM_{a}})\in \cohog{2}{\cM_{a}}(1)$, we have
\begin{equation*}
\cup \y_{\cM_{a}}^{i}: \Gr^{P}_{d-i}\cK_{a}\isom (\Gr^{P}_{d+i}\cK)[2i](i).
\end{equation*}
\end{lemma}
\begin{proof}
The fact that $\cup\y_{\cM}$ sends ${}^{p}\t_{\le i}\cK$ to $({}^{p}\t_{\le i+2}\cK)[2]$ is clear since $\cup\y_{\cM}: \cK\to \cK[2]$.
 
Let $\cA^{x}\subset \cA$ be the open subscheme consisting of those $a:X\to \frc/\Gm$ such that $a(x)\in \frc^{\rs}/\Gm$. By our assumption on the degree of $\cL$, the evaluation map $\ev_{x}: \cA\to  \frc/\Gm$ at $x$ is surjective, hence $\cA^{x}$ is open dense in $\cA$. In particular, $\cA^{\el,x}=\cA^{\el}\cap \cA^{x}$ is open dense in $\cA^{\el}$. 

By Theorem \ref{th:full supp}, each simple constituent of ${}^{p}\cK^{i}$ has full support. To show that \eqref{cup eta i global} is an isomorphism, it suffices to show that it is an isomorphism when restricted to $\cA^{\el,x}$. \footnote{In fact here we only need a weaker statement than full support: we only need that each simple constituent of ${}^{p}\cK^{i}$ has support that intersect $\cA^{\el,x}$, which is true for the perverse cohomology sheaves of the whole complex $\bR f_{*}\Qlbar$, as proved in \cite[Theorem 2.1.1]{YLD}.}.

The map $\nu: \cM\to \cM^{\Hit}$ restricts to a map $\nu^{\el, x}: \cM|_{\cA^{\el,x}}\to \cM^{\Hit}|_{\cA^{\el,x}}$ that is a $W$-torsor (since the choices of $B$-reductions compatible with $\ph(x)$, which is regular semisimple, form a $W$-torsor). In particular, $\nu^{\el,x}$ is finite. Since $\LL_{\Hit}$ is anti-ample on $\cM^{\Hit}|_{\cA^{\el,x}}$, its pullback to $\cM|_{\cA^{\el,x}}$ is also anti-ample. The isomorphism then follows from the relative hard Lefschetz theorem \cite[Th\'eor\`eme 5.4.10]{BBD}.
\end{proof}

\begin{lemma}\label{l:GrPH palindromic} For any $i\in \ZZ$ we have 
\begin{equation*}
\dim \Gr^{P}_{d_{\g}-i}H_{\g,\pur}= \dim \Gr^{P}_{d_{\g}+i}H_{\g,\pur}.
\end{equation*}
\end{lemma}
\begin{proof}
Define Laurent polynomials
\begin{eqnarray*}
&&\Phi(t)=\sum_{i\in \ZZ}(\dim \Gr^{P}_{d_{\g}+i}H_{\g,\pur}) t^{i}\\
&&\Psi(t)=\sum_{i\in \ZZ}(\dim \Gr^{P}_{d+i}\cK_{a,\pur}) t^{i}\\
&&A(t)=\sum_{i\in \ZZ}(\dim \homog{\ov d-i}{A_{a}})t^{i}=t^{-\ov d}(1+t)^{2\ov d}.
\end{eqnarray*}
By the freeness statement of Proposition \ref{p:perv Gr free} and \eqref{GrP isom}, we have
\begin{equation*}
\Psi(t)=\Phi(t)A(t).
\end{equation*}
Now $\Psi(t)$ satisfies $\Psi(t)=\Psi(t^{-1})$ by Lemma \ref{l:global HL} restricted to the pure part. Clearly $A(t)=A(t^{-1})$. We conclude that $\Phi(t)=\Phi(t^{-1})$, which is what we need.
\end{proof}


Let $\y=c_{1}(\LL_{\Gr})\in \cohog{2}{\Gr}(1)$. Note that $\y$ in Definition \ref{d:perv fil} lies in $\cohog{2}{\Gr}$, while here we added in a Tate twist. We also use $\y$ denote its restriction to $\Gr_{\g}$ or pullback to $\Fl_{\g}$.

\begin{prop}\label{p:HL} The perverse filtration on $H_{\g,\pur}$ satisfies hard Lefschetz with respect to $\y=c_{1}(\LL_{\Gr})\in \cohog{2}{\Gr_{G}}(1)$. More precisely, cupping with $\y$ takes $P_i H_{\g,\pur}$ to $P_{i+2} H_{\g,\pur}(1)$, and induces an isomorphism for every $0\le i\le d_\g$
\begin{equation}\label{cup eta i}
\cup \y^{i}: \Gr^P_{d_\g- i} H_{\g,\pur}\isom (\Gr^P_{d_\g+i}H_{\g,\pur})[2i](i).
\end{equation}  
\end{prop}
\begin{proof} We already showed in Lemma \ref{l:cup any local} that $\cup\y$ takes $P_i H_{\g,\pur}$ to $P_{i+2} H_{\g,\pur}$ for all $i$. We now show that \eqref{cup eta i} is an isomorphism by induction on $i$. Let $\y_{\cA_{a}}=c_{1}(\LL_{\cP_{a}}|_{\cP'_{a}})\in \cohog{2}{\cP'_{a}}\cong \cohog{2}{A_{a}}$. By Lemma \ref{l:det line factor}, under the isomorphism $\a^{*}$ in Lemma \ref{l:Ka tensor prod}, we have
\begin{equation}\label{Chern class sum}
\y_{\cM_{a}}=\y_{\cA_{a}}\ot1+1\ot \y.
\end{equation}
By Lemma \ref{l:GrPH palindromic} it suffices to show that $\cup \y^{i}: \Gr^P_{d_\g-i} H_{\g,\pur}\isom (\Gr^P_{d_\g+i}H_{\g,\pur})[2i]$ is injective. 

Now let $0\ne \z\in \cohog{\ov d}{A_{a}}$ be a primitive class with respect to $\y_{\cA_{a}}$, i.e., $\y_{\cA_{a}}\cup \z=0$. It exists because $\dim\cohog{\ov d}{A_{a}}=\binom{2\ov d}{\ov d}>\binom{2\ov d}{\ov d+2}=\dim \cohog{\ov d+2}{A_{a}}$. Consider the following diagram
\begin{equation}\label{cup eta comm}
\xymatrix{\Gr^P_{d_\g- i} H_{\g,\pur}\ar[r]^-{\cup \y^{i}}\ar[d]_{\z\ot(-)} & (\Gr^P_{d_\g+i} H_{\g,\pur})[2i](i)\ar[d]^{\z\ot(-)}\\
\cohog{d}{A_{a}}\ot \Gr^P_{d_\g- i} H_{\g,\pur}\ar@{^{(}->}[d]\ar[r]^-{\id\ot\y^{i}}& \cohog{d}{A_{a}}\ot (\Gr^P_{d_\g+i} H_{\g,\pur})[2i](i)\ar@{^{(}->}[d]\\
\Gr^P_{d- i} \cK_{a,\pur}\ar[r]^-{\cup \y_{\cM_{a}}^{i}}  & (\Gr^P_{d+i} \cK_{a,\pur})[2i](i)}
\end{equation}
The two hooked inclusion maps are from the decomposition in Proposition \ref{p:GrP tensor prod}(2). This diagram is commutative since for any $\xi\in \Gr^P_{d_\g- i} H_{\g,\pur}$,
\begin{equation*}
\y_{\cM_{a}}^{i}\cup (\z\ot \xi)=(\y_{\cA_{a}}\ot 1+1\ot \y)^{i}\cup (\z\ot \xi)=\z\ot (\y^{i}\cup \xi)
\end{equation*}
using \eqref{Chern class sum} and the fact that $\y_{\cA_{a}}\cup\z=0$.  Since the left vertical maps in \eqref{cup eta comm} are injective, and the bottom horizontal map is an isomorphism by Lemma \ref{l:global HL}, the top horizontal map is also injective. This finishes the proof. 
\end{proof}

\begin{theorem}\label{th:perv fil}
The filtration $P_{\bu}H_{\g,\pur}$ on $H_{\g,\pur}$ is a perverse filtration in the sense of Definition \ref{d:perv fil}.
\end{theorem}
\begin{proof} 
We check the defining properties for the perverse filtration.

\eqref{comp grading} holds by construction.

\eqref{fil bounds} By \eqref{non-can decomp}, the range of the perverse filtration $P_{\bu}\cK_{a}$ is the interval $[0,2d]$. Therefore $P_{2d}\cK_{a}=\cK_{a}$ hence $P_{2d}\cK_{a,\pur}=\cK_{a,\pur}$. This implies $P_{2d_{\g}}H_{\g,\pur}=H_{\g,\pur}$, hence $\Gr^{P}_{i}H_{\g,\pur}=0$ for $i>2d_{\g}$. By Lemma \ref{l:GrPH palindromic}, we have $\Gr^{P}_{i}H_{\g,\pur}=0$ for $i<0$. Since $P_{i'}H_{\g,\pur}=0$ for $i'$ sufficiently negative, this implies $P_{-1}H_{\g,\pur}=0$.

We now show that $P_{i}H^{i}_{\g,\pur}=H^{i}_{\g,\pur}$. By construction, it suffices to show that $P_{i}(\upH^{i}\cK_{a,\pur})=\upH^{i}\cK_{a,\pur}$, or even stronger $P_{i}(\upH^{i}\cK_{a})=\upH^{i}\cK_{a}$, or $\Gr^{P}_{i'}(\upH^{i}\cK_{a})=0$ for $i'>i$. In other words, we need to show that $\Gr^{P}_{i'}\cK_{a}$ is in cohomological degrees $\ge i'$. By \eqref{non-can decomp}, $\Gr^{P}_{i'}\cK_{a}$ is the stalk of a shifted perverse sheaf ${}^{p}\cK^{\dim \cA+i'}[-\dim \cA-i']$ at $a$. By Theorem \ref{th:full supp}, ${}^{p}\cK^{\dim \cA+i'}$ has full support. Therefore, all stalks of ${}^{p}\cK^{\dim \cA+i'}$ lie in degrees $\ge -\dim \cA$, and hence all stalks of ${}^{p}\cK^{\dim \cA+i'}[-\dim \cA-i']$ lie in degrees $\ge i'$. 

\eqref{PW} By Lemma \ref{l:Wa equiv} and the definition of $P_{i}H_{\g,\pur}$ as the image of perverse filtration pieces of $\cK_{a,\pur}$, it suffices to show that $P_{i}\cK_{a,\pur}$, or $P_{i}\cK_{a}$, is stable under $\Wa$. However, this is clear because $\Wa$ acts on the complex $\cK$, and it has to preserve the perverse truncations ${}^{p}\t_{\le i}\cK$.

\eqref{HL} is proved in Proposition \ref{p:HL}.

\eqref{Chern half} Each $\l\in \xch(T)$ gives a $G$-equivariant line bundles $\cL_{\l}$ on $\cB$. We need to show that the cup product by $c_{1}(\cL_{\l})$ on $H_{\g,\pur}$ sends $P_{i}H_{\g,\pur}$ to $P_{i+1}H_{\g,\pur}$. 

We have maps
\begin{equation*}
M_{a}\xr{\pi} \L_{\g}\bs \Fl_{\g}\to \BB T
\end{equation*}
taking $(\cE, \cE^{B}_{x}, \ph)$ to the $T$-torsor $\cE^{B}_{x}/N$, and this map factors through $\L_{\g}\bs \Fl_{\g}$. The $G$-equivariant line bundle $\cL_{\l}$ on $\cB$ descends to $\BB T$; we use the same notation $\cL_{\l}$ to denote its pullback to $ \L_{\g}\bs \Fl_{\g}$ and to $\cM_{a}$. It is clear that the map $\t$ is equivariant under the cup product by $c_{1}(\cL_{\l})$. Therefore, by the definition of the perverse filtration on $H_{\g,\pur}$, it suffices to show that $c_{1}(\cL_{\l})$ sends $P_{ i}\cK_{a,\pur}$ to $P_{i+1}\cK_{a,\pur}$. By \cite[Lemma 3.2.3]{YLD}, $c_{1}(\cL_{\l})$ sends $P_{ i}\cK_{a}$ to $P_{i+1}\cK_{a}$, hence the same is true for the pure part.
\end{proof}

\begin{cor}\label{c:perv fil exists} For any regular semisimple and integral $\g_{0}\in\frg\lr{t}$, $H_{\g_{0},\pur}$ carries a perverse filtration in the sense of  Definition \ref{d:perv fil}.
\end{cor}
\begin{proof} With the globalization data $(X,\cL, x, a)$ satisfying the statement of Lemma \ref{l:approx gamma}, we transport the filtration $P_{\bu}H_{\g,\pur}$ to $H_{\g_{0},\pur}$ under the isomorphism $H_{\g_{0},\pur}\cong H_{\g,\pur}$ as in \eqref{isom coho Flg}. By Lemma \ref{l:approx equiv}, this isomorphism is equivariant under the actions of $\Sym(\cohog{2}{\Fl})$ and $\Wa$, Theorem \ref{th:perv fil} implies that $H_{\g_{0},\pur}$ also satisfies all the axioms for a perverse filtration.
\end{proof}

\begin{remark} It is likely that the perverse filtration on $H_{\g_{0},\pur}$ is independent of the choice of globalization data $(X,\cL, x, a)$.  We offer two pieces of evidence:
\begin{enumerate}
\item In the case $G=\SL_{n}$, one can interpret $\Gr_{\g}$ as a local version of the compactified Jacobian of a planar curve singularity. In \cite{MY} and \cite{MS}, a perverse filtration was defined on the global compactified Jacobian of a planar curve using a versal deformation, which was shown to be independent of the choice of versal deformations in \cite[Proposition 2.15]{MY}.
\item As we will see in Lemma \ref{l:perv strict comp}, the perverse filtration intersected with the image of the map \eqref{intro ph} is independent of any choices.
\end{enumerate}\end{remark}

\section{Top cohomology of $\Fl_\g$ versus total cohomology of $\cB_e$}

\subsection{Reduction types} 


We recall the notion of (minimal) reduction types of a topologically nilpotent $\g\in \frg\lr{t}$ from \cite{Y-RT}. We have the evaluation map
\begin{equation*}
\ev_\g: \Gr_{\g}\to [\cN/G]
\end{equation*}
sending $g L^+G\in \Gr_\g$ to the nilpotent orbit of $\Ad(g^{-1})\g$ mod $t$. For each nilpotent orbit $\cO$ in $\frg$, let $\Gr_{\g,\cO}\subset \Gr_{\g}$ be the preimage of $[\cO/G]$ under $\ev_\g$. The {\em reduction type} $\RT(\g)$ of $\g$ is the set of nilpotent orbits $\cO$ such that $\Gr_{\g, \cO}\ne\vn$. The {\em minimal reduction type $\RT_{\min}(\g)$ } of $\g$ is the set of minimal elements in $\RT(\g)$.

For any nilpotent orbit $\cO$ and $e\in \cO$, let $\cB_e\subset \cB$ be its Springer fiber and let $d_{\cO}=\dim \cB_{e}$. The component group $A_e=\pi_0(G_e)$ of the centralizer $G_{e}$ acts on $\cohog{*}{\cB_e}$, whose invariants we denote by
\begin{equation*}
H_{\cO}=\cohog{*}{\cB_e}^{A_e}.
\end{equation*}
As the notation suggests, $H_{\cO}$ as a graded $W$-module is canonically independent of the choice of $e$. Denote the cohomologically graded pieces of $H_{\cO}$ by $H^{j}_{\cO}$. The inclusion $i_e: \cB_e\incl \cB$ induces a restriction map on cohomology 
\begin{equation*}
i_e^*: \cohog{*}{\cB}\to H_{\cO}
\end{equation*}
whose image we denote by $'H_{\cO}$. The top degree piece of $'H_{\cO}$ is the irreducible $W$-module $E_{\cO}$ corresponding to the nilpotent orbit $\cO$ and the trivial local system under Springer correspondence.

 
\subsection{$E_{\g}$ as a $W$-module}
Consider the projection map 
\begin{equation*}
\pi_\g: \Fl_\g\to \Gr_\g.
\end{equation*}
Let $\Fl_{\g,\cO}=\pi_{\g}^{-1}(\Gr_{\g,\cO})\subset \Fl_{\g}$. Then $\pi_{\g}$ restricts to a fibration $\pi_{\g,\cO}: \Fl_{\g,\cO}\to \Gr_{\g,\cO}$ whose fibers are isomorphic to $\cB_{e}$ (for $e\in \cO$). From this we get the inequality \eqref{dim GrO}. We call $\cO\in \RT(\g)$ {\em $\g$-relevant} if the following equivalent conditions are satisfied:
\begin{enumerate}
\item $\dim\Gr_{\g,\cO}=d_{\g}-d_{\cO}$.
\item $\dim \Fl_{\g,\cO}=d_{\g}$.
\item $\Fl_{\g,\cO}$ contains an open subset of $\Fl_{\g}$.
\end{enumerate}
The equivalence between (1) and (2) is clear; the equivalence between (2) and (3) follows from the fact that $\Fl_{\g}$ is equidimensional (\cite[\S4, Proposition 1]{KL}). 

We denote by $\RT_{\rel}(\g)\subset \RT(\g)$ the set of $\g$-relevant nilpotent orbits. For each $\cO\in \RT_{\rel}(\g)$, choose $e\in \cO$ to identify $\cO/G$ with $\{e\}/G_{e}$. Define $\wt\Gr_{\g,e}$ by the Cartesian diagram
\begin{equation*}
\xymatrix{\wt\Gr_{\g,e}\ar[r]\ar[d]^{\nu_{\g,\cO}} & \{e\}/G_{e}^{\c}\ar[d]\\
\Gr_{\g,\cO}\ar[r]^-{\ev_{\g,\cO}} & \cO/G\cong \{e\}/G_{e}}
\end{equation*}
By construction, $\nu_{\g,\cO}: \wt\Gr_{\g,e}\to \Gr_{\g,\cO}$ is a right $A_{e}=\pi_{0}(G_{e})$-torsor. Moreover, $\wt\Gr_{\g,e}$ can be identified with the sub-ind-scheme of $LG/((L^{+}G)^{1}G^{\c}_{e})$ (where $(L^{+}G)^{1}=\ker(L^{+}G\to G)$) consisting of $g\in LG/((L^{+}G)^{1}G^{\c}_{e})$ such that $\Ad(g^{-1})\g\in e+t\frg\tl{t}$. From this description, we see that the $LG_{\g}$-action on $\Gr_{\g,\cO}$ lifts to $\wt\Gr_{\g,e}$, commuting with the $A_{e}$-action.

\begin{lemma}\label{l:decomp top}
As an $A_{\g}\times W$-module, we have a canonical isomorphism
\begin{equation}\label{decomp top}
\homog{2d_{\g}}{\Fl_{\g}}\cong \bigoplus_{\cO\in \RT_{\rel}(\g)}(\hBM{2d_{\g}-2d_{\cO}}{\wt\Gr_{\g, e}}\ot \homog{2d_{\cO}}{\cB_{e}})_{A_{e}}.
\end{equation}
Here we choose an element $e\in \cO$ for each $\cO\in \RT_{\rel}(\g)$. The  $A_{e}$-coinvariants on the right side is taken with respect to the diagonal action. The $A_{\g}$-action on the right side is on the $\hBM{2d_{\g}-2d_{\cO}}{\wt\Gr_{\g, e}}$ factor, and $W$ acts on the $\homog{2d_{\cO}}{\cB_{e}}$ factor.
\end{lemma}
\begin{proof} The ind-scheme $\Fl_{\g}$ has a stratification $\Fl_{\g,\cO}$ indexed by $\cO\in \RT(\g)$. Let $\wh\Gr_{\g,e}\to \Gr_{\g,\cO}$ be the total space of the $G_{e}$-torsor classified by the evaluation map $\Gr_{\g,\cO}\to \cO/G\cong \{e\}/G_{e}$. In other words, $\wh\Gr_{\g,e}=\{g\in LG/(L^{+}G)^{1}|\Ad(g^{-1})\g\in e+t\frg\tl{t}\}$. Then the multiplication map $LG/(L^{+}G)^{1}\times G/B\to LG/\bI=\Fl$ restricts to a map $\wh\Gr_{\g,e}\times \cB_{e}\to \Fl_{\g,\cO}$, inducing an isomorphism
\begin{equation*}
\Fl_{\g,\cO}\cong (\wh\Gr_{\g,e}\times \cB_{e})/G_{e}.
\end{equation*}
Note that $\wh\Gr_{\g,e}\to \wt\Gr_{\g,e}$ is a $G_{e}^{\c}$-torsor, therefore $\dim\wh\Gr_{\g,e}\le d_{\g}-d_{\cO}+\dim G_{e}=d_{\g}+d_{\cO}+r$ (where $r$ is the rank of $G$), with equality if and only if $\cO$ is $\g$-relevant. Taking top homology we get
\begin{equation*}
\hBM{2d_{\g}}{\Fl_{\g,\cO}}\cong (\hBM{2(d_{\g}+d_{\cO}+r)}{\wh\Gr_{\g, e}}\ot \homog{2d_{\cO}}{\cB_{e}})_{A_{e}}.
\end{equation*}
The $G_{e}^{\c}$-torsor $\wh\Gr_{\g,e}\to \wt\Gr_{\g,e}$ induces an $A_{\g}\times A_{e}$-equivariant isomorphism of top homology groups, therefore we can rewrite the above as
\begin{equation}\label{top hom Fl gO}
\hBM{2d_{\g}}{\Fl_{\g,\cO}}\cong (\hBM{2d_{\g}-2d_{\cO}}{\wt\Gr_{\g, e}}\ot \homog{2d_{\cO}}{\cB_{e}})_{A_{e}}.
\end{equation}
Finally we have a canonical decomposition
\begin{equation*}
\hBM{2d_{\g}}{\Fl_{\g}}=\bigoplus_{\cO\in \RT_{\rel}(\g)}\hBM{2d_{\g}}{\Fl_{\g,\cO}}
\end{equation*}
because both the left side and each term on the right are spanned by the fundamental classes of top-dimensional irreducible components. Moreover, this decomposition is equivariant under $A_{\g}\times W$. Combined with \eqref{top hom Fl gO} we get \eqref{decomp top}.

\end{proof}
 
Taking $A_{\g}$-coinvariants in \eqref{decomp top},  we get:
\begin{cor}
We have a canonical isomorphism of $W$-modules
\begin{equation}\label{decomp top Mg}
\homog{2d_{\g}}{\Fl_{\g}}_{A_{\g}}\cong \bigoplus_{\cO\in \RT_{\rel}(\g)}(\hBM{2d_{\g}-2d_{\cO}}{\wt\Gr_{\g, e}}_{A_{\g}}\ot \homog{2d_{\cO}}{\cB_{e}})_{A_{e}}.
\end{equation}
\end{cor} 
 
\begin{cor}\label{c:rel crit} Let $\cO$ be a nilpotent orbit and $e\in \cO$. Then $\cO\in \RT_{\rel}(\g)$ if and only if $E_{\cO}$ appears as a summand of $E_{\g}$ as a $W$-module. Moreover, the dual to the multiplicity space of $E_{\cO}$ in $E_{\g}$ can be canonically identified with
\begin{equation*}
\hBM{2d_{\g}-2d_{\cO}}{\Gr_{\g, \cO}}_{A_{\g}}.
\end{equation*}
\end{cor}
\begin{proof}
In the decomposition \eqref{decomp top Mg}, all simple $W$-modules isomorphic to $E_{\cO}\cong\homog{2d_{\cO}}{\cB_{e}}_{A_{e}}$ has to come from the summand corresponding to $\cO$, because by Springer correspondence, for another nilpotent orbits $\cO'\notin \cO$ and $e'\in \cO'$, $\homog{2d_{\cO'}}{\cB_{e'}}$ does not contain subquotients isomorphic to $E_{\cO}$ as a $W$-module. The multiplicity  space of $E_{\cO}$ in $\homog{2d_{\g}}{\Fl_{\g}}_{A_{\g}}$, which is the dual of $E_{\g}$, is $\hBM{2d_{\g}-2d_{\cO}}{\wt\Gr_{\g, e}}_{A_{\g}\times A_{e}}\cong \hBM{2d_{\g}-2d_{\cO}}{\Gr_{\g, \cO}}_{A_{\g}}$, using that $\wt \Gr_{\g,e}\to \Gr_{\g,\cO}$ is an $A_{e}$-torsor.
\end{proof}

\subsection{Perverse filtration on Chern class polynomials}
Consider the composition
\begin{equation}\label{restr Sym to Flg}
\ph_{\g}: S:=\Sym^{*}(\cohog{2}{\Fl})\to \cohog{*}{\Fl}\xr{i^{*}_{\g}} \cohog{*}{\Fl_{\g}}^{A_{\g}}=H_{\g}\surj H_{\g,\pur}
\end{equation}
where the first map is given by the ring structure on $ \cohog{*}{\Fl}$.  Now we define a ``perverse'' grading on the symmetric algebra $S:=\Sym^{*}(\cohog{2}{\Fl})$ as follows. Using the decomposition
\begin{equation*}
\cohog{2}{\Fl}=\cohog{2}{\cB}\op \Qlbar \y
\end{equation*}
we can write
\begin{equation*}
S\cong \Sym^{*}(\cohog{2}{\cB})\ot \Qlbar[\y].
\end{equation*}
The perverse grading assigns degree $1$ to elements in $\cohog{2}{\cB}$, and assigns degree $2$ to $\y$, and is required to be additive under multiplication. In other words, the perverse degree $n$ piece of $S$ is defined to be
\begin{equation}\label{nS}
{}_{n}S:=\bigoplus_{j+2i=n}\Sym^{j}(\cohog{2}{\cB})\ot \y^{i}.
\end{equation}
Define a perverse filtration $P_{\bu}S$ on $S$ by
\begin{equation*}
P_{n}S:=\bigoplus_{n'\le n}{}_{n'}S=\bigoplus_{j+2i\le n}\Sym^{j}(\cohog{2}{\cB})\ot \y^{i}.
\end{equation*}

\begin{lemma}\label{l:perv strict comp} The map $\ph_{\g}$ in \eqref{restr Sym to Flg} is {\em strictly} compatible with the perverse filtrations on $S$ and on $H_{\g,\pur}$. In other words, for each $i\in \ZZ$, $\ph_{\g}(S)\cap P_{i}H_{\g,\pur} =\ph_{\g}(P_{i}S)$.
\end{lemma}
\begin{proof}
By properties \eqref{Chern half} and \eqref{HL} of a perverse filtration on $H_{\g,\pur}$, $\ph_{\g}$ sends 
$\Sym^{j}(\cohog{2}{\cB})\y^{i}$ to $P_{j+2i}H^{2j+2i}_{\g,\pur}$. Therefore $\ph_{\g}(P_{i}S)\subset P_{i}H_{\g,\pur}$.

To show strictness, we may restrict to each cohomological degree $2n$, $0\le n\le d_{\g}$. Let $0\ne \xi=\sum_{j=0}^{n}\th_{j}\y^{n-j}\in \Sym^{n}(\cohog{2}{\Fl})$, where $\th_{j}\in \Sym^{j}(\cohog{2}{\cB})$. Suppose $\ph_{\g}(\xi)\in P_{i}H^{2n}_{\g,\pur}$, we need to show that there exists $\xi'\in P_{i}S$ such that $\ph_{\g}(\xi)=\ph_{\g}(\xi')$. We take $\xi'=\sum_{j\ge 2n-i}\th_{j}\y^{n-j}$, which clearly belongs to $P_{i}S$. We need to show that $\ph_{\g}(\xi-\xi')=\sum_{j<2n-i}\ph_{\g}(\th_{j})\y^{n-j}=0$. Let $j_{0}$ be the smallest integer $j$ such that $\ph_{\g}(\th_{j})\ne 0$ (as an element in $H^{2j}_{\g,\pur}$). Then $\ph_{\g}(\xi)\subset P_{2n-j_{0}}H^{2n}_{\g,\pur}$. Projecting to $\Gr^{P}_{2n-j_{0}}H^{2n}_{\g,\pur}$, only the term $\th_{j_{0}}\y^{n-j_{0}}$ in $\xi$ contributes. Therefore the image of $\ph_{\g}(\xi)$ in  $\Gr^{P}_{2n-j_{0}}H^{2n}_{\g,\pur}$ is the same as the image of $\ph_{\g}(\th_{j_{0}})$ under the map
\begin{equation*}
\cup \y^{n-j_{0}}: \Gr^{P}_{j_{0}}H^{2j_{0}}_{\g,\pur}\to \Gr^{P}_{2n-j_{0}}H^{2n}_{\g,\pur}.
\end{equation*}
Since $n\le d_{\g}$, the above map is injective by property \eqref{HL} of the perverse filtration. Now $\ph_{\g}(\th_{j_{0}})\in P_{j_{0}}H^{2j_{0}}_{\g,\pur}\isom \Gr^{P}_{j_{0}}H^{2j_{0}}_{\g,\pur}$ is nonzero by assumption, hence the image of $\ph_{\g}(\xi)$ in $\Gr^{P}_{2n-j_{0}}H^{2n}_{\g,\pur}$ is nonzero. Since we assume $\ph_{\g}(\xi)\in P_{i}H^{2n}_{\g,\pur}$, this forces $2n-j_{0}\le i$, i.e., $j_{0}\ge 2n-i$, which implies $\ph_{\g}(\xi)=\ph_{\g}(\xi')$. 
\end{proof}

Consider the degree $2d_{\r}$ part of $\ph_{\g}$
\begin{equation}\label{restr Sym to Flg top deg}
\ph^{d_{\g}}_{\g}: \Sym^{d_{\g}}(\cohog{2}{\Fl})\to \cohog{2d_{\g}}{\Fl}\xr{i^{*}_{\g}} \cohog{2d_{\g}}{\Fl_{\g}}^{A_{\g}}=E_{\g}.
\end{equation}
Let
\begin{equation*}
'E_{\g}\subset E_{\g}
\end{equation*}
be the image of $\ph^{d_{\g}}_{\g}$. 
 
Let $P_\bu H_{\g,\pur}$ be a perverse filtration on $H_{\g,\pur}$, which exists by Corollary \ref{c:perv fil exists}. We denote by $P_\bu E_{\g}$ the induced filtration on the top degree  $E_{\g}=H^{2d_{\g}}_{\g}$, and let $P_{\bu}{}'E_{\g}={}'E_{\g}\cap P_{\bu}E_{\g}$. As we noted in the Introduction, the filtration $P_{\bu}{}'E_{\g}$ can be defined using the map $\ph^{d_{\g}}_{\g}$ and the explicit ``perverse grading'' on $S^{d_{\g}}$, without going through the constructions of \S\ref{s:perv fil}. 


\begin{theorem}\label{th:main}
Let $\cO\in \RT(\g)$. Then there is a canonical surjection of graded $W$-modules
\begin{equation*}
\th_{\g,\cO}: \Gr^{P}_{2d_{\g}-\bu}{}'E_{\g}\surj {}'H^{2\bu}_{\cO}.
\end{equation*}
In other words, for each $j\ge0$, there is a canonical surjection of $W$-modules
\begin{equation*}
\Gr^{P}_{2d_{\g}-j}{}'E_{\g}\surj {}'H^{2j}_{\cO}.
\end{equation*}
\end{theorem}
\begin{proof}
Recall the image ${}'E_{\g}=\Im(\ph^{d_{\g}}_{\g})\subset E_{\g}$ is equipped with the induced perverse filtration from $E_{\g}$. By the strictness proved in Lemma \ref{l:perv strict comp}, we have a canonical isomorphism of graded vector spaces
\begin{equation}\label{Gr E'}
\Gr^{P}_{\bu}{}'E_{\g}\cong \Im(\Gr^{P}_{\bu}S^{d_{\g}}\to \Gr^{P}_{\bu}E_{\g}).
\end{equation}
By \eqref{nS}, 
\begin{equation*}
\Gr^{P}_{2d_{\g}-j}S^{d_{\g}}=\Sym^{j}(\cohog{2}{\cB})\ot\y^{d_{\g}-j}.
\end{equation*}
Let
\begin{equation*}
\psi^{j}_{\g}:\Sym^{j}(\cohog{2}{\cB})\to P_{j}H^{2j}_{\g,\pur}=\Gr^{P}_{j} H^{2j}_{\g,\pur}
\end{equation*}
be the restriction of $\ph^{j}_{\g}$. We have a commutative diagram
\begin{equation*}
\xymatrix{\Sym^{j}(\cohog{2}{\cB})\ar[d]_{\psi^{j}_{\g}}\ar[r]^-{\y^{d_{\g}-j}}_-{\sim} & \Sym^{j}(\cohog{2}{\cB})\ot\y^{d_{\g}-j}\ar[d]^{\Gr^{P}_{2d_{\g}-j}\ph_{\g}^{d_{\g}}}\\
\Gr^{P}_{j}H^{2j}_{\g,\pur}\ar[r]^-{\y^{d_{\g}-j}} & \Gr^{P}_{2d_{\g}-j}E_{\g}}
\end{equation*}
where the bottom horizontal arrow is an isomorphism by the property \eqref{HL} of a perverse filtration. Therefore, $\y^{d_{\g}-j}$ induces a $W$-equivariant isomorphism 
\begin{equation*}
\Im(\psi^{j}_{\g})\cong \Im(\Gr^{P}_{2d_{\g}-j}\ph^{d_{\g}}_{\g})\stackrel{\eqref{Gr E'}}{=}\Gr^{P}_{2d_{\g}-j}{}'E_{\g}.
\end{equation*}
It remains to construct a $W$-equivariant surjection
\begin{equation*}
\Im(\psi^{j}_{\g})\surj {}'H^{2j}_{\cO}.
\end{equation*}
 
Pick a point $x=gL^{+}G\in \Gr_{\g,\cO}$, and let $e\in \cO$ be the reduction of $\Ad(g)^{-1}\g$ mod $t$. Then the fiber $\pi_\g^{-1}(x)\subset \Fl_\g$ is isomorphic to $\cB_e$ (well-defined up to $G_e$). 
Restriction along the inclusion $\k_x: \cB_e\cong \pi_\g^{-1}(x)\incl \Fl_\g$ gives a $W$-equivariant map
\begin{equation*}
\k_x^*: \cohog{*}{\Fl_\g}\to \cohog{*}{\cB_e}
\end{equation*}
and it factors through $\cohog{*}{\Fl_\g}_{\pur}$ because $\cohog{*}{\cB_e}$ is pure by Springer \cite{Spr}. Composing $\psi^{j}_{\g}$ with $\k_{x}^{*}$ we get a map of 
$W$-modules
\begin{equation}\label{res Flg to Be inv}
\Sym^{j}(\cohog{2}{\cB})\surj \cohog{2j}{\cB}\to P_jH^{2j}_{\g,\pur}\to \cohog{2j}{\cB_e}
\end{equation}
which is the restriction along the inclusion $i_e: \cB_e\incl \cB$. Therefore we get a surjection of $W$-modules
\begin{equation}\label{Im phj}
\Im(\psi^{j}_\g)\surj \Im(i_{e}^{*}: \cohog{2j}{\cB}\to \cohog{2j}{\cB_e})={}'H^{2j}_{\cO},
\end{equation}
as desired. Since $'H^{2j}_{\cO}$ as a quotient of $\Sym^{j}(\cohog{2}{\cB})$ does not depend on the choice of $e\in \cO$, the map \eqref{Im phj} does not depend on the choice of $x\in\Gr_{\g,\cO}$ either, hence the canonicity of the map.
\end{proof}

\begin{cor}\label{c:all rel} All nilpotent orbits $\cO\in \RT(\g)$ are $\g$-relevant, i.e., $\RT_{\rel}(\g)=\RT(\g)$.
\end{cor}
\begin{proof}
Let $\cO\in \RT(\g)$. By Theorem \ref{th:main}, $E_{\cO}$ is a subquotient of $E_{\g}$ as a $W$-module. By Corollary \ref{c:rel crit}, $\cO$ is $\g$-relevant. 
\end{proof}

\begin{cor} If $\cO\in \RT(\g)$ and $\cO'\ge \cO$, then $\cO'\in \RT(\g)$. In other words, the subset
\begin{equation*}
\bigcup_{\cO\in \RT(\g)}\cO
\end{equation*}
is open in $\cN$.
\end{cor}

%
\begin{proof} Let $\cO$ be a nilpotent orbit such that $\Gr_{\g,\cO}\ne\vn$, and let $\cO'$ be another nilpotent orbit  such that  $\ov{\cO'}\supset\cO$. By Lemma \ref{c:rel crit}, it suffices to show that $E_{\cO'}$ appears in $E_{\g}$ as a $W$-module. By Theorem \ref{th:main}, it suffices to show that ${}'H_{\cO}$ contains $E_{\cO'}$ as a $W$-module.

%

Let $e\in \cO$ and $e'\in \cO'$. Let $d$ and $d'$ be the dimensions of $\cB_{e}$ and $\cB_{e'}$. Consider the Springer map $\pi: \wt\cN\to \cN$. Then $i^{*}_{e}$ is the cospecialization map from the stalk of $\pi_{*}\Qlbar$ at $0$ to the stalk at $e$. We may  compose this with the further cospecialization map to the stalk at $e'$ and get
\begin{equation*}
i^{*}_{e'}: \cohog{*}{\cB}\xr{i^{*}_{e}} \cohog{*}{\cB_{e}}\xr{s}  \cohog{*}{\cB_{e'}}
\end{equation*}
Taking degree $2d'$ we get
\begin{equation*}
i^{*}_{e'}: \cohog{2d'}{\cB}\xr{i^{*}_{e}} \cohog{2d'}{\cB_{e}}\xr{s} \cohog{2d'}{\cB_{e'}}\cong E_{\cO'}\oplus (\cdots)
\end{equation*}
We know that the image of $i^{*}_{e'}$ in degree $2d'$ is precisely  $E_{\cO'}$. Therefore ${}'H^{2d'}_{\cO}=\Im(i^{*}_{e})\cap \cohog{2d'}{\cB_{e}}$ maps surjectively to $E_{\cO'}$.
\end{proof}


\begin{exam}
Consider the case where $d_{\g}=1$ (the subregular affine Springer fibers). In this case, the description of $\Fl_{\g}$ in \cite[\S7.6-7.8]{KL} shows that $\cohog{2}{\Fl}\to \cohog{2}{\Fl_{\g}}^{A_{\g}}=E_{\g}$ is an isomorphism. On the other hand, $\RT(\g)$ consists of the regular and subregular orbits. Let $\cO$ be the subregular nilpotent orbit of $\frg$. Then ${}'H_{\cO}$ is precisely $\cohog{0}{\cB}\op\cohog{2}{\cB}$, which has the same dimension as $ \cohog{2}{\Fl}$. By Theorem \ref{th:main}, we conclude that in this case $\th_{\g,\cO}$ is an isomorphism. Of course this isomorphism can be constructed directly from the explicit configuration of projective lines in $\Fl_{\g}$ and $\cB_{e}$. 
\end{exam}

\begin{remark} If $\cO\in \RT(\g)$ and $\cO\le \cO'$, then $'H_{\cO'}$ is a further quotient of $'H_{\cO}$, both as quotients of $\cohog{*}{\cB}$. The canonicity of the map $\th_{\g,\cO}$ in Theorem \ref{th:main} shows that there is a commutative diagram
\begin{equation*}
\xymatrix{\Gr^{P}_{2d_{\g}-\bu}{}'E_{\g} \ar@{->>}[r]^-{\th_{\g,\cO}} \ar@/_1pc/@{->>}[rr]_-{\th_{\g,\cO'}} & {}'H^{2\bu}_{\cO}\ar@{->>}[r] & {}'H^{2\bu}_{\cO'}}
\end{equation*}
Therefore, the maps $\th_{g,\cO}$ for the minimal reduction types $\cO$ of $\g$ (conjecturally there is only one minimal reduction type $\cO$) determine the maps $\th_{g,\cO'}$ for all reduction types $\cO'$.
\end{remark}

\subsection{A refinement of a theorem of Tsai}

When $\g\in \frg\lr{t}$ is such that $t^{-1}\g$ is integral, C-C.Tsai proved the following result on the number of irreducible components of $\Fl_{\g}$.

\begin{theorem}[Tsai \cite{Tsai}]\label{th:Tsai}  Suppose $\g\in \frg\lr{t}$ is regular semisimple and $t^{-1}\g$ is integral.  Then 
\begin{equation*}
\dim E_{\g}=|W|.
\end{equation*}
\end{theorem}
Tsai's method is via counting points on $\Fl_{\g}$ over $\FF_{q}$. Such point-counting is well-known to be equal to orbital integrals, and he then uses Shalika germ expansion to show that the dominant contribution comes from the regular nilpotent orbit, giving $|W|$.


Using Theorem \ref{th:main}, we get the following strengthening of Tsai's theorem.

\begin{prop}\label{p:reg rep} Suppose $\g\in \frg\lr{t}$ is regular semisimple and $t^{-1}\g$ is integral. 
\begin{enumerate}
\item The restriction map $\ph^{d_{\g}}_{\g}$ in \eqref{restr Sym to Flg top deg} is surjective (i.e., $'E_{\g}=E_{\g}$). 
\item The map $\th_{\g,\{0\}}$ is an isomorphism of $W$-modules
\begin{equation*}
\Gr^{P}_{2d_{\g}-\bu}E_{\g}\isom \cohog{2\bu}{\cB}.
\end{equation*}
In particular, as a $W$-module, $\Gr^{P}_{*}E_{\g}$ is isomorphic to the regular $W$-module.
\item The perverse filtration on $E_{\g}$ admits a canonical splitting (i.e., a perverse grading), with the degree $2d_{\g}-j$ piece given by the image of ${}_{2d_{\g}-j}S^{d_{\g}}$ under $\ph^{d_{\g}}$, which is isomorphic to $\cohog{2j}{\cB}\ot \y^{d_{\g}-j}$, for $0\le j\le N=\dim \cB$. 
\item As a $\Wa$-module, $E_{\g}$ is {\em dual} to the space of $W$-harmonic polynomials on $\cohog{2}{\cB}\cong \xch(T)_{\Qlbar}$, where the lattice part of $\Wa$ acts by translation via $\io: \xcoch(T)\incl \xch(T)_{\Qlbar}$ given by a $W$-invariant bilinear form on $\xcoch(T)$.
\end{enumerate}
\end{prop}
\begin{proof} 

(1) (2) The fact $t^{-1}\g$ is integral is equivalent that the adjoint orbit of $\g$ intersects $t\frg\tl{t}$, i.e.,  $\RT(\g)$ contains the zero orbit. By Theorem \ref{th:main}, $\Gr^{P}_{\bu}{}'E_{\g}$ maps surjectively onto $\Im(\cohog{*}{\cB}\to \cohog{*}{\cB})=\cohog{*}{\cB}$ as a $W$-module. Now $\dim \cohog{*}{\cB}=|W|$ and $\dim E_{\g}=|W|$ by Theorem \ref{th:Tsai}. Therefore, we must have ${}'E_{\g}=E_{\g}$, i.e., $\ph^{d_{\g}}_{\g}$ is surjective. Moreover, $\th_{\g,\{0\}}$ has to be an isomorphism for dimension reasons.

(3) Write $V=\cohog{2}{\cB}$ and $N=\dim \cB$. Since $t^{-1}\g$ is integral, we have $d_{\g}\ge N$. We identify $\Sym^{d_{\g}}(V)$ with $\op_{j=0}^{N}\Sym^{j}(V)\ot \y^{d_{\g}-j}$. Let $I=\ker(\Sym(V)\to \cohog{*}{\cB})$ and $I^{j}=I\cap \Sym^{j}(V)$ be the homogeneous pieces of $I$. Since the map $\ph_{\g}$ factors through $\cohog{*}{\cB}\ot \Ql[\y]$, we see that $\op_{j=0}^{N}I^{j}\ot \y^{d_{\g}-j}\subset \ker(\ph^{2d_{\g}}_{\g})$. Since $\dim E_{\g}=|W|$,  for dimension reasons, we must have the equality $\ker(\ph^{2d_{\g}}_{\g})=\bigoplus_{j=0}^{N}I^{j}\ot \y^{d_{\g}-j}$. This identifies $E_{\g}$ with the following quotient of $\Sym^{d_{\g}}(V\op \Qlbar \y)$
\begin{equation}\label{Eg as quot}
E_{\g}\cong \Sym^{d_{\g}}(V\op \Qlbar \y)/\left(\bigoplus_{j=0}^{N}I^{j}\ot \y^{d_{\g}-j}\right)\cong \bigoplus_{j=0}^{N}\cohog{2j}{\cB}\ot \y^{d_{\g}-j} \cong \Sym(V)/I.
\end{equation}
This gives the required splitting of the perverse filtration.

(4)The canonical action of $\Wa$ on $\cohog{*}{\Fl}$ when restricted to $\cohog{2}{\Fl}\cong V\op \Qlbar \y$ takes the following form: $W$ acts on $V$ by the reflection representation and acts trivially on $\y$; $\l\in \xcoch(T)$ acts trivially on $V$ and sends $\y$ to $\y+\io_{\y}(\l)$, where $\io_{\y}: \xcoch(T)\to V\cong \xch(T)_{\Qlbar}$ is induced from the $W$-invariant bilinear form on $\xcoch(T)$ given by $\y$. This $\Wa$-action induces a $\Wa$-action on the quotient $\Sym(V)/I$ of $\Sym(\cohog{2}{\Fl})$, making \eqref{Eg as quot} $\Wa$-equivariant.

View $\Sym(V)$ as the ring of differential operators on the polynomial ring $\Qlbar[V]\cong \Sym(V^{*})$ that are invariant under translations. Consider the graded-perfect pairing
\begin{equation*}
\Sym(V)\times \Qlbar[V]\to \Qlbar
\end{equation*}
defined by $(D, f)=(Df)(0)$, where $D\in \Sym(V)$ and $f\in \Qlbar[V]$. Recall the space of $W$-harmonic polynomials $\cH\subset \Qlbar[V]$ consists of those annihilated by the $W$-invariant differential operators $\Sym^{>0}(V)^{W}$, which generate $I$. Therefore, restricting the second argument of the above pairing to $\cH$, the pairing factors through a perfect pairing
\begin{equation}\label{pairing harmonic}
\Sym(V)/I\times \cH\to \Qlbar
\end{equation}
Viewing $V$ as the affine space $V+\y\subset V\op \Qlbar\y$, the $\Wa$-action on the latter restricts to an action of $\Wa$ on $V$ by affine linear transformations: $W$ acts by the reflection representation and $\l\in \xcoch(T)$ acts by translation by $\io_{\y}(\l)$. From the construction, the pairing \eqref{pairing harmonic} is $\Wa$-equivariant. By the $\Wa$-equivariant isomorphism \eqref{Eg as quot}, we conclude that $E_{\g}$ is dual to $\cH$ as $\Wa$-modules.

%
%



\end{proof}

\subsection{Description of $E_{\g}$ in type $A$}\label{ss:A}
Let $G=\SL_{n}$. We know from \cite[Theorem 1.18]{Y-RT} that every regular semisimple and topologically nilpotent $\g\in\frg\lr{t}$ has a unique minimal reduction type $\RT_{\min}(\g)$. Let us describe the Jordan type of $\cO$. 

Let $f(x)\in F[x]$ be the characteristic polynomial of $\g$, and $f(x)=\prod_{i\in I}f_{i}(x)$ be its factorization into monic irreducible polynomials over $F$. Let $\val: \ov F^{\times}\to \QQ$ be the extension of the discrete valuation on $F$ such that $\val(t)=1$. Then the roots of each $f_{i}$ have the same valuation $v_{i}\in \QQ_{>0}$, whose denominator is $\deg f_{i}$. For each positive rational number $q=\frac{a}{b}$ in lowest terms, we define a partition $\l_{q}$ of $b$ as follows: if $q<1$, then $\l_{q}$ has exactly $a$ parts and each part is either $\lfloor \frac{b}{a}\rfloor$ or $\lceil \frac{b}{a}\rceil$; if $q\ge1$ then $\l_{q}$ is the trivial partition (all $1$s). In other words, $\l_{a/b}$ is the most balanced partition of $b$ with $\min\{a,b\}$ parts. 

\begin{lemma}\label{l:min red A}
The Jordan type of the minimal reduction type of $\g$ is the partition $\l(\g)$ of $n$ given by concatenating $\{\l_{v_{i}}\}_{i\in I}$.
\end{lemma}
\begin{proof}
The factorization pattern of $f$ implies that the adjoint orbit of $\g$ meets a Levi subgroup $M\subset G$ whose block sizes are given by $\{\deg f_{i}\}_{i\in I}$. We may assume $\g\in \fm\lr{t}$. By \cite[Corollary 3.4]{Y-RT}, the minimal reduction type $\cO$ of $\g$ viewed as an element of $\frg\lr{t}$ is the $G$-orbit of the minimal reduction type $\cO_{M}$ of $\g$ viewed as an element of $\fm\lr{t}$. This allows us to reduce to the case where $f$ is irreducible, i.e., $\g$ is elliptic.

We compute the minimal reduction type $\cO$ assuming $\g$ is elliptic. Let $E=F[x]/(f(x))$, a degree $n$ field extension of $F$. Let $v=\frac{d}{n}$ be the valuation of a root of $f$, so that $(d,n)=1$. 

If $v\ge1$, then $t^{-1}\g$ is integral. The minimal reduction type is the zero orbit, and its Jordan type is the trivial partition, which is the same as $\l_{v}$ by definition. 

Consider the case $v<1$. In \cite[Definition 7.1]{Y-RT} we defined the notion of the skeleton of $\Gr_{G}$ with respect to a maximal $F$-torus $T$ in $G$: it is the fixed point locus of the neutral component of $LT$ on $\Gr_{G}$. In our case, since $G_{\g}$ is a maximal torus of Coxeter type, by \cite[Proposition 7.6(1)]{Y-RT}, the skeleton of $\Gr_{G}$ with respect to $G_{\g}$ is a single point. If we identify $\Gr_{G}$ with the moduli space of $\cO_{F}$-lattices in $E$ that have the same volume as $\cO_{E}$, then the lattice $\cO_{E}$ itself is the unique point in the skeleton. By \cite[Lemma 7.3]{Y-RT}, $\cO$ is the orbit of the nilpotent endomorphism $\ov\g: \cO_{E}/(t)\to  \cO_{E}/(t)$ (the reduction of multiplication by $\g$). Choosing a uniformizer $t^{1/n}$ of $E$ and using the $k$-basis $\{1,t^{1/n},\cdots, t^{(n-1)/n}\}$ of $\cO_{E}/(t)$, we see that the Jordan type of $\ov\g$ is the most balanced partition $\l$ of $n$ with $d$ parts, i.e., $\l_{v}$. 
\end{proof}

On the other hand, we consider the total cohomology $H_{\cO}$ of a nilpotent orbit $\cO$ of $\frg=\sl_{n}$. Suppose $M\subset G$ is a Levi subgroup and $\cO$ is the $G$-orbit of a nilpotent orbit $\cO_{M}$ of $\fm$. By the Alvis-Lusztig induction formula for total cohomology of Springer fibers \cite[Formula (e')]{AL}, we have
\begin{equation}\label{ind total Spr}
H_{\cO}\cong \Ind^{W}_{W_{M}}H_{\cO_{M}}.
\end{equation}
Here we use that $\pi_{0}(G_{e})$ acts trivially on the $\cohog{*}{\cB_{e}}$ so that $H_{\cO}=\cohog{*}{\cB_{e}}$  for all $e\in \cO$.

Now consider the minimal reduction type $\cO$ of $\g$. By Lemma \ref{l:min red A}, $\cO$ intersects the regular class $\cO_{M_{\l(\g)},\reg}$ in a Levi subgroup $M_{\l(\g)}\subset G$ whose block sizes are given by $\l(\g)$. Since $H_{\cO_{M_{\l(\g)},\reg}}$ is the trivial representation of $W_{M_{\l(\g)}}=W_{\l(\g)}$, we conclude from \eqref{ind total Spr} that
\begin{equation}\label{total Spr A}
H_{\cO}\cong\Ind^{W}_{W_{\l(\g)}}\Qlbar
\end{equation}
as $W$-modules. 


The following result is essentially due to Kivinen--Tsai \cite{KT}.

\begin{theorem}\label{th:A} Let $G=\SL_{n}$ and $\g\in\frg\lr{t}$ be a regular semisimple and topologically nilpotent element. Let $\cO$ be the minimal reduction type of $\g$, whose Jordan type is $\l(\g)$ described in \S\ref{ss:A}. Then the map $\ph^{d_{\g}}_{\g}$ in \eqref{restr Sym to Flg top deg} is surjective, and $\th_{\g,\cO}$ is an isomorphism of graded $W$-modules
\begin{equation*}
\Gr^{P}_{2d_{\g}-\bu}E_{\g}\cong H^{2\bu}_{\cO}.
\end{equation*}
As $W$-modules, both sides above are isomorphic to $\Ind^{W}_{W_{\l(\g)}}\Qlbar$.
\end{theorem}
\begin{proof}
In view of the surjectivity of $\th_{\g,\cO}$ proved in Theorem \ref{th:main}, it suffices to show that $\dim E_{\g}=\dim H_{\cO}$.

We first reduce the general case to the case where $\g$ is elliptic. Changing $\g$ in its adjoint orbit if necessary, there is a Levi subgroup $M\subset G$ whose block sizes are the degrees $\{\deg f_{i}\}$ of irreducible factors of the characteristic polynomial of $\g$, such that $\g\in \fm\lr{t}$ and is elliptic there. By Lemma \ref{l:ASF in Levi}, $\dim E_{G,\g}=|W/W_{M}|\dim E_{M,\g}$. On the other hand, by \cite[Corollary 3.4]{Y-RT}, the minimal reduction type $\cO$ of $\g$ viewed as an element of $\frg\lr{t}$ is the $G$-orbit of the minimal reduction type $\cO_{M}$ of $\g$ viewed as an element of $\fm\lr{t}$. By \eqref{ind total Spr}, we have $\dim  H_{\cO}=|W/W_{M}|\dim H_{\cO_{M}}$. It then suffices to show that $\dim E_{M,\g} = \dim H_{\cO_{M}}$. The neutral component of $\Fl_{M,\g}$ is isomorphic to a product of $\Fl_{\SL_{n_{i}},\g_{i}}$ where $n_{i}=\deg f_{i}$ are the block sizes of $M$, and $\g_{i}\in \fm_{i}\lr{t}$ has irreducible characteristic polynomial $f_{i}$. We thus reduce to the case where $\g$ is elliptic, which we now assume.

Let $v$ be the valuation of any root of $f(x)$, the characteristic polynomial of $\g$. By Lemma \ref{l:min red A}, the minimal reduction type $\cO$ of $\g$ has Jordan type $\l_{v}$. The result of Kivinen and Tsai \cite[Remark 8.6 and Theorem 8.8]{KT} says that $\dim E_{\g}=|W/W_{\l_{v}}|$ in this case. On the other hand, \eqref{total Spr A} implies $\dim H_{\cO}=|W/W_{\l_{v}}|$. The equality $\dim E_{\g}=\dim H_{\cO}$ is proved.
\end{proof}

\begin{lemma}\label{l:ASF in Levi}
Let $G$ be a reductive group and $M\subset G$ a Levi subgroup. Let $\g\in \fm\lr{t}$ be an element that is regular semisimple and integral as an element of $\frg\lr{t}$. Denote by $\Fl_{G,\g}$ and $\Fl_{M,\g}$ the affine Springer fibers of $\g$ in $\Fl_{G}$ and $\Fl_{M}$, and define $E_{G,\g}$ and $E_{M,\g}$ accordingly. Then
\begin{equation}\label{eq E}
\dim E_{G,\g}=|W/W_{M}|\dim E_{M,\g}.
\end{equation}
\end{lemma}
\begin{proof}
Choose a parabolic subgroup $P\subset G$ containing $M$ as a Levi factor. We may assume $P$ is standard, i.e., $P\supset B$. Let $N_{P}$ be the unipotent radical of $P$. For each connected component $(LM)^{\om}$ of $LM$, where $\om\in \pi_{1}(M)$, and each $w\in  [W_{M}\bs W]$ (minimal length representative with respect to the length function on $W$ given by $B$), consider the semi-infinite orbit $\Sig^{\om}_{w}:=(LN_{P})(LM)^{\om} w \bI/\bI\subset \Fl$.  The intersection $\Fl_{G,\g}\cap \Sig^{\om}_{w}$ has a canonical map $q^{\om}_{w}: \Fl_{G,\g}\cap \Sig^{\om}_{w}\to \Fl_{M,\g}\cap \Fl^{\om}_{M}$ (where $\Fl^{\om}_{M}=(LM)^{\om}/\bI_{M}$) by sending $umw\bI$ to $m\bI_{M}$. The same argument as \cite[\S5, Proposition 1]{KL} shows that $q^{\om}_{w}$ is an iterated affine space fibration, and that $\Fl_{G,\g}\cap \Sig^{\om}_{w}\to \Fl_{M,\g}$ has the same dimension as $\Fl_{G,\g}$. Let $\Sig_{w}=(LP) w\bI/\bI\subset \Fl$, which is the union of $\Sig^{\om}_{w}$ for varying $\om\in \pi_{1}(M)$.  The maps $q^{\om}_{w}$ induce a bijection between the sets of irreducible components $\frq_{w}: \Irr(\Fl_{G,\g}\cap \Sig_{w})\isom  \Irr(\Fl_{M,\g})$ that is equivariant under the translation actions by $A_{\g}$. It then induces a bijection on $A_{\g}$-orbits $\ov\frq_{w}: \Irr(\Fl_{G,\g}\cap \Sig_{w})/A_{\g}\isom  \Irr(\Fl_{M,\g})/A_{\g}$.
Now $\{\Sig_{w}\}_{w\in [W_{M}\bs W]}$ form a partition of $\Fl_{G}$, we conclude that $|\Irr(\Fl_{G,\g})/A_{\g}|=|W/W_{M}|\cdot|\Irr(\Fl_{M,\g})/A_{\g}|$, which then implies \eqref{eq E}.\end{proof}

\quash{
It is natural to make the following conjecture.
\begin{conj} Suppose $G$ is simply-connected. Let $\g\in\frg\lr{t}$ be a topologically nilpotent element. Then the map
\begin{equation*}
\Sym^{d}(\cohog{2}{\Fl_{G}})\to \cohog{2d}{\Fl_{\g}}^{A_{\g}\rtimes B_{\g}}
\end{equation*}
is surjective. Here, there is a monodromy action of a certain braid group $B_{\g}$ coming from the space of $\g$ with the same root valuation data, and the map should land in the invariants of $A_{\g}\rtimes B_{\g}$.
\end{conj}

}


\begin{thebibliography}{99}

\bibitem{AL}
D.Alvis, G. Lusztig,
On Springer's correspondence for simple groups of type $E_{n}$   ($n=6,7,8$). 
Math. Proc. Cambridge Philos. Soc. 92 (1982), no. 1, 65--78.

\bibitem{BBD}
A. Beilinson, J.Bernstein, P.Deligne, O.Gabber,
Faisceaux pervers.
Ast\'erisque 2018, no. 100, vi+180 pp.


\bibitem{D-HodgeII}
P.Deligne,
Th\'eorie de Hodge, II.
Publications Math\'ematiques de l'IH\'ES. 40 (1971) pp. 5--58.


\bibitem{D-poids}
P.Deligne,
Poids dans la cohomologie des vari\'et\'es alg\'ebriques.
Actes du congr\`es international des math\'ematiciens (Vancouver). (1974) pp. 79--85.

\bibitem{F}
G.Faltings,
Stable G-bundles and projective connections.
J. Algebraic Geom. 2 (1993), no. 3, 507--568.


\bibitem{KL}
D.Kazhdan, G. Lusztig,
Fixed point varieties on affine flag manifolds.
Israel J. Math. 62 (1988), no. 2, 129--168.

\bibitem{KT}
O.Kivinen, C-C.Tsai,
Shalika germs for tamely ramified elements in $\GL_n$, arXiv:2209.02509.

\bibitem{LN}
G.Laumon, B-C. Ng\^o, 
Le lemme fondamental pour les groupes unitaires.
Ann. of Math. (2) 168 (2008), no. 2, 477--573.


\bibitem{L}
G.Lusztig, 
Affine Weyl groups and conjugacy classes in Weyl groups.
Transform. Groups 1 (1996), no. 1-2, 83--97.


\bibitem{LComm}
G.Lusztig, 
Comments to my papers, 
arxiv:1707.09368.

\bibitem{MY}
D. Maulik, Z. Yun,
Macdonald formula for curves with planar singularities.
J. Reine Angew. Math. 694 (2014), 27--48.


\bibitem{MS}
L.Migliorini, V. Shende,
A support theorem for Hilbert schemes of planar curves.
J. Eur. Math. Soc. 15 (2013), no. 6, 2353--2367.

\bibitem{NgoFL}
B-C. Ng\^o, 
Le lemme fondamental pour les alg\`ebres de Lie.
Publ. Math. Inst. Hautes \'Etudes Sci. No. 111 (2010), 1--169.

\bibitem{OY}
A. Oblomkov, Z. Yun,
Geometric representations of graded and rational Cherednik algebras.
Adv. Math. 292 (2016), 601--706.

\bibitem{Spa}
N.Spaltenstein,
Classes unipotentes et sous-groupes de Borel.
Lecture Notes in Math., 946. Springer-Verlag, Berlin-New York, 1982, ix+259 pp.

\bibitem{Spr}
T. A. Springer,
A purity result for fixed point varieties in flag manifolds.
J. Fac. Sci. Univ. Tokyo Sect. IA Math. 31 (1984), no. 2, 271--282.


\bibitem{Tsai}
C-C. Tsai,
Components of affine Springer fibers.
Int. Math. Res. Not. IMRN 2020, no. 6, 1882--1919. 

\bibitem{YGS}
Z.Yun, Global Springer theory.
Advances in Math. 228 (2011), 266--328.

\bibitem{YLD}
Z.Yun, Langlands duality and global Springer theory.
Compos. Math. 148 (2012), no. 3, 835--867.

\bibitem{YSph}
Z.Yun, Spherical part of local and global Springer actions
Math. Ann. 359 (2014), no. 3-4, 557--594.

\bibitem{Y-RT}
Z.Yun, Minimal reduction types and the Kazhdan-Lusztig map.
Indag. Math. (N.S.) 32 (2021), no. 6, 1240--1274.


\end{thebibliography}
\end{document}